\documentclass[reqno]{amsart}
\usepackage[left=4cm,right=4cm,
    top=2cm,bottom=2cm,bindingoffset=0cm]{geometry}
\usepackage{float}
\usepackage{cite}
\usepackage{dsfont}
\usepackage{amsmath}
\usepackage{graphicx}
\usepackage{latexsym}
\usepackage{amsfonts}
\usepackage{amssymb}
\usepackage{color}
\usepackage{xcolor}
\usepackage{algorithm}
\usepackage{nicematrix}
\usepackage{algpseudocode}
\usepackage{float}
\usepackage{caption}
\def\R{\mathbb{R}}
\def\C{\mathbb{C}}
\def\E{\mathbb{E}}

\def\P{\mathbb{P}}

\def\N{\mathbb{N}}
\def\Z{\mathbb{Z}}

\def\K{\mathcal{K}}

\def\deg{\mathrm{deg}\:}

\newtheorem{formula}{}
\newtheorem{lemma}[formula]{\indent Lemma}
\newtheorem{theorem}[formula]{\indent Theorem}
\newtheorem{corollary}[formula]{\indent Corollary}

\newtheorem{proposition}[formula]{\indent Proposition}
\newtheorem{remark}[formula]{\indent Remark}
\newtheorem{example}[formula]{\indent Example}
\newtheorem{definition}[formula]{\indent Definition}

\newcommand{\revprod}{\mathop{\overleftarrow{\prod}}}

\begin{document}
\title[$2$-dimensional random walks in cones]
{Harmonic polynomials and other exactly computable characteristics for $2$-dimensional random walks in cones} 
\thanks{Funded by the Deutsche Forschungsgemeinschaft (DFG, German Research Foundation) –
Project-ID 317210226 – SFB 1283.}
\author[Denisov]{Denis Denisov}
\address{Department of Mathematics, University of Manchester, UK}
\email{denis.denisov@manchester.ac.uk}
\author[Elizarov]{Nikita Elizarov}
\address{Faculty of Mathematics, Bielefeld University, Germany}
\email{nelizarov@math.uni-bielefeld.de}

\author[Wachtel]{Vitali Wachtel}
\address{Faculty of Mathematics, Bielefeld University, Germany}
\email{wachtel@math.uni-bielefeld.de}

\begin{abstract}
        In this note we consider $2$-dimensional lattice random walks killed at leaving a wedge with opening $\alpha\in(0,\pi]$. Assuming that the walk cannot jump over the boundary of the wedge we prove that there exists a harmonic polynomial if and only if $\alpha=\pi/m$ with some integer $m$. Our proof is constructive and allows one to give exact expressions for harmonic polynomials for every integer $m$. Furthermore, we give exact expressions for all finite moments of the exit time, this result is valid for all angles $\alpha$.
\end{abstract}


\maketitle
\section{Introduction and main results}
Let $S(n)=X(1)+X(2)+\ldots+X(n)$, $n\ge1$ be a $d$-dimensional random walk, where $\{X(k)\}$ are independent copies of the vector $X=(X_1,X_2,\ldots,X_d)$. We shall always assume that 
\begin{equation}
\label{eq:1_and_2-moments}
\E[X_i]=0\text{ and }\E[X_iX_j]=\delta_{ij}\text{ for all }1\le i,j\le d.
\end{equation}
Let $K$ be an open cone in $\R^d$ with vertex at zero. For every $x\in K$ we define 
$$
\tau_x:=\{n\ge1: x+S(n)\notin K\}.
$$
In papers \cite{DW15,DW19,DW24}, the authors have shown the existence of positive harmonic functions for the walk $S(n)$ killed at leaving the cone $K$. More 
precisely, these papers contain different constructions of a positive function $V(x)$ satisfying 
\begin{equation}
\label{eq:V-harm}
V(x)=\E[V(x+S(1));\tau_x>1],\quad x\in K.
\end{equation}
Besides the existence of a solution to \eqref{eq:V-harm}, one knows some asymptotic bounds for $V$, which imply that 
$$
V(x)\sim u(x)\quad\text{as }\delta(x):={\rm dist}(x,\partial K)\to\infty,
$$
where $u(x)$ is harmonic on $K$ in the classical sense, that is, $u(x)>0$ in the cone $K$
and 
\begin{equation}
\label{eq:u-harm}
\Delta u(x)=0,\ x\in K
\quad\text{and}\quad 
u\vert_{\partial K}=0. 
\end{equation}
But there are almost no examples where one has a closed form expression for the function $V$.
The purpose of the present note is to find a class of examples where one can compute $V(x)$ explicitly.

We shall primary deal with the case $d = 2$ and a family of convex cones
$\{K_\alpha\}$, which are described by their boundaries:
\begin{align*}
        x_2 = 0,\quad x_2 = \tan{(\alpha)}x_1,\quad \alpha\in(0,\pi].
\end{align*}
For every cone $K_\alpha$ one has a closed form expression for solutions to \eqref{eq:u-harm}:
all multiples of the function
$$
u_\alpha(r\cos\varphi,r\sin\varphi):=r^{\pi/\alpha}\sin\left(\pi\varphi/\alpha\right),
\quad r>0,\:\:\varphi\in(0,\alpha)
$$
solve \eqref{eq:u-harm} with $K=K_\alpha$. This implies that if $\alpha=\pi/m$ for some integer $m$, then $u_{\pi/m}(x)$ is a homogeneous polynomial of degree $m$. For example,
\begin{equation}
\label{eq:u_for_m=3}
u_{\pi/3}(x_1,x_2)=3x_1^2x_2-x_2^3
\end{equation}
and
\begin{equation}
\label{eq:u_for_m=4}
u_{\pi/4}(x)=x_1^3x_2-x_1x_2^3.
\end{equation}
We next notice that the functions $u_{\pi/m}$ remain harmonic in the whole plane:
$$
\Delta u_{\pi/m}(x)=0\quad\text{for all }x\in\R^2.
$$
But these functions are not harmonic for discrete time random walks in the whole plane. Simple calculations show that 
$$
\E[u_{\pi/3}(x+X)]=u_{\pi/3}(x)+3\E[X_1^2X_2]-\E[X_2^3]
$$
for any random walk satisfying \eqref{eq:1_and_2-moments} and having 
finite third moments. Thus, the harmonicity property holds if and only if 
$3\E[X_1^2X_2]=\E[X_2^3]$. However, one can obtain a harmonic polynomial 
by adding a polynomial of degree $2$. Indeed, for a polynomial 
$Ax_1^2+Bx_2^2+Cx_1x_2$ one has 
\begin{align*}
    &\E[A(x_1+X_1)^2+B(x_2+X_2)^2+C(x_1+X_1)(x_2+X_2)]\\
    &\hspace{1cm}=Ax_1^2+Bx_2^2+A+B+Cx_1x_2
\end{align*}
and choosing constants $A$, $B$ and $C$ so that $A+B=-3\E[X_1^2X_2]+\E[X_2^3]$, we conclude that the polynomial $h(x)=u_{\pi/3}(x)+Ax_1^2+Bx_2^2+Cx_1x_2$ satisfies
\begin{equation}
\label{eq:harm-wholeplane}
\E[h(x+X)]=h(x),\quad x\in\mathbb{R}^2.
\end{equation}
Furthermore, it is easy to check that if one chooses 
$$
A=0,\quad B=-3\E[X_1^2X_2]+\E[X_2^3]\quad\text{and}\quad 
C=3^{3/2}\E[X_1^2X_2]-3^{1/2}\E[X_2^3],
$$
then, besides \eqref{eq:harm-wholeplane}, one has $h(x)=0$ for every $x\in\partial K_{\pi/3}$.
Thus, if we consider a walk $S(n)$ such that $x+S(\tau_x)\in\partial K_{\pi/3}$ then, applying the optional stopping theorem to the martingale $h(x+S(n))$, we obtain 
\begin{align*}
\E[h(x+S(n));\tau_x>n]
=\E[h(x+S(n\wedge\tau_x))]=h(x).
\end{align*}
In particular, the function $h$ is harmonic for $S(n)$ killed at leaving $K_{\pi/3}$. 
Our first result shows that this observation remains valid for all
wedges $K_{\pi/m}$, $m\ge1$.
\begin{theorem}
\label{thm:harmonic}
Assume that the random walk $S(n)$ is a lattice, i.e. $\P(S(n)\in R\text{ for all }n)=1$ for some $2$-dimensional lattice $R$.  Assume also that
\eqref{eq:1_and_2-moments} holds and that 
\begin{equation}
\label{eq:no-overshoot}
x+S(\tau_x)\in\partial K_\alpha\quad\text{for all }x\in R\cap\partial K_\alpha.
\end{equation}
\begin{itemize}
    \item[(a)] If $\alpha=\pi/m$ for some $m\ge2$ and if $\E|X|^m$ is finite then there exists a polynomial $r_m(x)$ of degree at most $(m-1)$ such that 
    $h(x):=u_{\pi/m}(x)+r_m(x)$ is harmonic for $S(n)$ killed at leaving $K_{\pi/m}$.
    All the coefficients of $r_m(x)$ are linear combinations of mixed moments $\E[X_1^kX_2^l]$,
    $k,l\le m$.
    \item[(b)] If a positive polynomial $h(x)$ has degree $m$, is harmonic for $S(n)$ killed at leaving $K_\alpha$, $\alpha\in(0,\pi)$ and 
    $h(x)\vert_{R\cap\partial K_\alpha}=0$ then
    $\alpha=\pi/m$ and there exists a constant $c$ such that the degree of $h(x)-cu_{\pi/m}(x)$
    does not exceed $m-1$.
\end{itemize}
\end{theorem}
\begin{remark}
Zero boundary conditions imply that the polynomial $h(x)$ is divisible by $x_2(x_2-x_1\tan\alpha)$.
\hfill $\diamond$
\end{remark}
One of the main reasons for us to look at wedges $K_\alpha$ is the connection with combinatorics of lattice walks in the positive quadrant. Let $\{Y(n)\}$ be independent copies of a $\Z^2$-valued vector $Y$ with $\E Y=0$ and with finite second moments. Consider a random walk $U(n)=Y(1)+Y(2)+\ldots+Y(n)$ which is killed at leaving the positive quadrant $\R^2_+$. Let $\varrho$ denote the correlation coefficient of the coordinates $Y_1$ and $Y_2$, that is,
$$
\varrho=\frac{{\rm Cov}(Y_1,Y_2)}{\sqrt{\E Y_1^2\E Y_2^2}}\in(-1,1).
$$
(We exclude here the degenerate case $|\varrho|=1$.) Set
$$
T=
\left(
\begin{array}{cc}
\frac{1}{\sqrt{\E Y_1^2(1-\varrho^2)}}, &\frac{-\varrho}{\sqrt{\E Y_2^2(1-\rho^2)}}\\
0, &\frac{1}{\sqrt{\E Y_2^2}}
\end{array}
\right).
$$
It is then easy to see that the vector $X=TY$ satisfies \eqref{eq:1_and_2-moments}.
Furthermore,
$$
T\R_+^2=K_{\alpha(\varrho)}
\quad\text{with}\quad
\alpha(\varrho)=\arctan\left(\frac{\sqrt{1-\varrho^2}}{\varrho}\right).
$$
Therefore, we can apply first Theorem~\ref{thm:harmonic} to the walk
$S(n)=TU(n)$ killed at leaving $K_{\alpha(\varrho)}$. If there exists a harmonic polynomial for $S(n)$ then we can apply $T^{-1}$ to get a harmonic polynomial for $U(n)$.
We also notice that the condition \eqref{eq:no-overshoot} is fulfilled if the vector $Y$ satisfies the condition 
$$
\P(Y_1\ge-1,Y_2\ge-1)=1,
$$
which is a natural generalisation of the classical left-continuity property for one-dimensional random walks. 

At the first glance, our condition \eqref{eq:no-overshoot} looks very restrictive. But looking at one-dimensional walks, one can see that this condition is rather natural. In the one-dimensional case one has only two cones: $(0,\infty$ and $(-\infty,0)$. For definiteness we consider $K=(0,\infty)$. It is well-known that if $S(n)$ is a zero mean $1$-dimensional walk then the harmonic function for $S(n)$ killed at leaving $K$ is proportional to the renewal function  corresponding to weak descending ladder heights. Assume that this renewal function $H$ is linear, that is, $H(x)=ax+b$ for all $x=0,1,2,\ldots$. (Notice that the subadditivity of the renewal function implies that $b\ge a$.) Then, the generating function of its differences is simple to compute:
$$
\sum_{x=0}^\infty s^x(H(x)-H(x-1))
=b+a\frac{s}{1-s}.
$$
Let $g(s)$ denote the generating function of the first weak descending ladder height. Then 
$$
\frac{1}{1-g(s)}=b+a\frac{s}{1-s}.
$$
Solving this equation, we obtain 
$$
g(s)=\frac{(b-1)+(a+1-b)s}{b-(b-a)s}.
$$
It is now easy to see that there are only two following possibilities:
\begin{itemize}
\item if $b>a$ then the first ladder height has shifted geometric distribution,
\item if $b=a$ then $g(s)=\frac{b-1}{b}+\frac{1}{b}s$ and, consequently, the first ladder height has only two possible values.
\end{itemize}
The first case corresponds to the geometric distribution of the negative part of $X$ and the second one to the left-continuous random walk. So, the left-continuous walk is one of two cases when one gets a linear harmonic function.
These simple arguments lead to the following question. Can one introduce an appropriate $2$-dimensional analogue of walks with geometric left tails?  At the moment we have no idea how to do that.

Theorem~\ref{thm:harmonic}(a) claims the existence of a harmonic polynomial which possesses, as one can easily conclude, the following property:
$$
h(x)\sim u_{\pi/m}(x)\quad\text{as }\delta(x)\to\infty.
$$
A positive harmonic function $V(x)$ constructed in \cite{DW24} has the same asymptotic property:
$$
V(x)\sim u_{\pi/m}(x)\quad\text{as }\delta(x)\to\infty.
$$
Thus, it is rather natural to expect that the harmonic functions $h(x)$ and $V(x)$
should be equal. This is confirmed by the following result.
\begin{theorem}
\label{thm:uniqueness}
Under the conditions of Theorem~\ref{thm:harmonic}(a),
$$
h(x)=V(x)\quad\text{ for all }x\in K_{\pi/m}\cap R.
$$
In particular,
$h(x)\ge0$ for all $x\in K_{\pi/m}\cap R.$
\end{theorem}
Results similar to our Theorems~\eqref{thm:harmonic} and~\eqref{thm:uniqueness} have recently been obtained by Hubert and Raschel in \cite{HR25}. They have proved claims in Theorem~\ref{thm:harmonic}(a) and in Theorem~\eqref{thm:uniqueness} by using a quite different method, which seems to be less constructive. Our approach is, on contrary, algorithmic and allows one to compute harmonic polynomials by using the standard methods from the linear algebra only. A further difference to \cite{HR25} are moment conditions imposed on the random variable $X$: Hubert and Raschel assume that $X$ has finite number of jumps and we require only moments of order $n$ when working in the cone $K_{\pi/n}$. It should also be mentioned that \cite{HR25} contains a much stronger statement than our Theorem~\ref{thm:harmonic}(b): it is proved there that if there exists an asymptotically positive harmonic polynomial for a random walk killed at leaving a $d$-dimensional cone $K$ then $K$ is a Weyl chamber of a finite Coxeter group.

The methods, which we have used to derive properties of harmonic functions, may be used to derive explicit expressions for the moments of exit times $\tau_x$.
\begin{theorem}
\label{thm:moments}
Assume that the random variable $X$ satisfies the conditions \eqref{eq:1_and_2-moments} and \eqref{eq:no-overshoot}. For every integer $k<\frac{\pi}{2\alpha}$, the function $x\mapsto\E[\tau_x^k]$ is a polynomial of degree $2k$. In particular, 
$$
\E[\tau_x]=x_2(x_1\tan\alpha-x_2).
$$
All higher moments can be computed recursively.
\end{theorem}
This result allows one to determine also the moments of $S(\tau_x)$. At the end of the paper we show how to derive explicit expressions for the first two moments of $S(\tau_x)$ from Theorem~\ref{thm:moments}.
\section{Proof of Theorem~\ref{thm:harmonic}}
\subsection{Preliminary results.}
We start by recalling some properties of positive harmonic functions for killed Brownian motions.
\begin{lemma}
\label{lem:BM1}
For every $\alpha\in(0,\pi)$, the function
$u_\alpha(x_1, x_2) := \Im(x_1 + i x_2)^\frac{\pi}{\alpha}$
is positive in $K_\alpha$ and harmonic for the Brownian motion killed at leaving 
$K_\alpha$, that is,
\begin{align}
\label{eq:BM1}
\Delta u_\alpha(x) = 0,
\quad x\in K_\alpha,\quad
u_\alpha\big|_{\partial K_\alpha} = 0.
\end{align}
Furthermore, for every $x\in K_\alpha$ one has
\begin{align}
\label{eq:BM2}
1\le\frac{u_{\alpha}(x)}{|x|^{\frac{\pi}{\alpha} - 1}\delta(x)}
\le \frac{\pi}{\alpha}.
\end{align}
\end{lemma}
\begin{proof}
We give a proof of these results just to have a self-contained presentation.  

As the function $z^\frac{\pi}{\alpha}$ in analytical on $\C\setminus\{0\}$,
$\Im z^\frac{\pi}{\alpha}$ and $\Re z^\frac{\pi}{\alpha}$ are harmonic in
the classical sense:
        \begin{align*}
                \Delta \Im (x_1+ix_2)^\frac{\pi}{\alpha} 
                = \Delta \Re(x_1+ix_2)^\frac{\pi}{\alpha}= 0.
        \end{align*}
        Moreover, for all $t\ge 0$,
        \begin{align*}
                \Im(t+i0)^\frac{\pi}{\alpha} = 0,
                \quad 
                \Im\left(t\cos{\alpha} + it\sin{\alpha}\right)^{\frac{\pi}{\alpha}} = 0.
        \end{align*}
        Consequently, \eqref{eq:BM1} holds. 
        We next notice that $u_{\alpha}(x) > 0$ for all $x\in K_{\alpha}.$ 
        Indeed, for  $x\in \K_{\alpha}$ one has $0 < \mathrm{arg}\:(x_1 + i x_2) < \alpha$ which leads to $0 < \frac{\pi}{\alpha} 
        \arg\: (x_1 + i x_2) < \pi.$ Thus, $u_{\alpha}$ is positive in $K_{\alpha}.$

To prove \eqref{eq:BM2} we write $u_{\alpha}(x)$ in polar coordinates: 
        \begin{align*}
                u_{\alpha}(x) = \Im(x_1 + i x_2)^\frac{\pi}{\alpha} = \Im(re^{i\beta})^\frac{\pi}{\alpha} = 
                r^\frac{\pi}{\alpha}\sin\left({\frac{\pi\beta}{\alpha}}\right),
        \end{align*}
        where $r = |x_1+ix_2|$ and $\beta = \arg{(x_1+ix_2)}.$
        Then 
        \begin{align*}
                \frac{u_{\alpha}(x_1,x_2)}{r^{\frac{\pi}{\alpha} - 1}\delta(x)} = \frac{r\sin\left(\frac{\pi\beta}{\alpha}\right)}
                {\delta(x)}.
        \end{align*}
        If $\beta \le \alpha/2$ then $\delta(x) = r\sin{\beta}$ and,
        consequently,
        \begin{align*}
                \frac{r\sin{\frac{\pi\beta}{\alpha}}}{\delta(x)} = \frac{\sin{\frac{\pi\beta}{\alpha}}}{\sin{\beta}}.
        \end{align*}
        For $0 \le \beta \le \alpha/2$ we have 
        $0 \le \frac{\pi\beta}{\alpha} \le \pi/2.$ Since $\sin\varphi$
        increases on $[0,\pi/2]$, 
        \begin{align*}
            \frac{\sin{\frac{\pi\beta}{\alpha}}}{\sin{\beta}} \ge 1. 
        \end{align*}
        To obtain the upper bound in \eqref{eq:BM2} for $\beta\le\alpha/2$, we notice that, due to the concavity of $\sin\varphi$ on $[0,\pi/2]$ for the line 
        $x_2 = \frac{\sin\beta}{\beta}x_1$,
        $$
        \sin\left(\frac{\pi}{\alpha}\beta\right)
        \le\frac{\pi}{\alpha}\sin\beta,\quad \beta\le\frac{\alpha}{2}.
        $$

        If $\beta \ge \alpha/2$ then $\delta(x) = r\sin{\bigl(\alpha - \beta\bigr)}.$ Consequently,
        \begin{align*}
                \frac{r\sin{\frac{\pi\beta}{\alpha}}}{r\sin{\bigl(\alpha - \beta\bigr)}} = 
                \frac{\sin{(\pi - \frac{\pi\beta}{\alpha})}}{\sin{\bigl(\alpha - \beta\bigr)}} =
                \frac{\sin{\frac{\pi}{\alpha}(\alpha - \beta)}}{\sin{\bigl(\alpha - \beta\bigr)}}.
        \end{align*}
        Setting now $\beta'=\alpha-\beta$ we reduce the problem to the already considered case $\beta\le\alpha/2$. Thus, the proof of the lemma is complete.
\end{proof}
\begin{lemma}\label{lem:f=zero}
    Let $f(x_1,x_2)$ be a polynomial. If $f(x)=0$ for all $x\in R\cap K_\alpha$ with some $\alpha\in(0,2\pi)$ then $f(x)=0$ for all $x\in\R^2$.
\end{lemma}
\begin{proof}
    Let $n$ denote the degree of $f$. We can represent $f$ as the sum of homogeneous polynomials:
    $$
    f=f_n+f_{n-1}+\ldots+f_0,
    $$
    where every $f_k$ is a homogeneous polynomial of degree $k$. 
    The boundary condition for $f$ implies that $f_k(x)=0$ for all 
    $x\in R\cap K_\alpha$ and for every $k$.
    
    Clearly, it suffices to show that $f_n\equiv0$. Writing $f_n$ in polar coordinates, we have 
    $$
    f_n(r\cos\beta,r\sin\beta)=r^n q_n(\beta).
    $$
    Thus, we need to show that $q_n\equiv0$. Assume that, on contrary, there exists $0 < \beta_0 < \alpha$ such that $q_{n}(\beta_0) \ne 0.$ Without loss of generality, we assume that $q_{n}(\beta_0) > 0$. Since $f_n$ is a polynomial,
    the function $q_n$ is continuous and, consequently, there exists $\varepsilon>0$ such that 
    $$
    q_n(\beta)>\frac{q_n(\beta_0)}{2}
    \quad\text{for all }\beta\in[\beta_0-\varepsilon,\beta_0+\varepsilon].
    $$
    Therefore, $f(r\cos\beta,r\sin\beta)\to\infty$ as $r\to\infty$ uniformly in 
    $\beta\in[\beta_0-\varepsilon,\beta_0+\varepsilon]$. This contradicts the observation $f_n(x)=0$ for all $x\in R\cap K_\alpha$. Thus, $f_n(x)=0$ for
    all $x\in K_\alpha$. Since $f_n$ is polynomial, we conclude that $f\equiv0$.
\end{proof}

For every random vector $X$ we define operator $L_X$ by letting
$$
L_X\phi(x):=\E[\phi(x+X)]-\phi(x),
$$
here we may plug any function $\phi$ such that $\E[|\phi(x+X)|]<\infty$ for all $x$.
\begin{lemma}\label{taylor_exp}
        Assume that $X$ satisfies \eqref{eq:1_and_2-moments} and that 
        $\E|X|^n$ is finite. Then for every polynomial $f$ of degree $n$ there exists a polynomial $g_f$ of degree $n-3$ such that 
        \begin{align*}
                L_Xf(x) = \frac{1}{2}\Delta f(x) + g_f(x),\quad x\in\R^2.
        \end{align*}
        The coefficients of $g_f$ are expectations of some 
        polynomials of $X_1$ and $X_2$.
\end{lemma}
\begin{proof}
        By the Taylor formula,
        \begin{align*}
                f(x_1 + X_1, &x_2 + X_2) - f(x_1, x_2) = \frac{\partial f}{\partial x_1}
                (x_1, x_2) X_1 + \frac{\partial f}{\partial x_2}(x_1, x_2) X_2 \\
                &+ \frac{1}{2} \left(\frac{\partial^2 f}{\partial x_1^2}(x_1, x_2) X_1^2 + 2 
                \frac{\partial^2 f}{\partial x_1 \partial x_2}(x_1, x_2) X_1 X_2 + \frac{\partial^2 f}
                {\partial x_2^2}(x_1, x_2) X_2^2 \right) 
                \\
                &+ \sum\limits_{k + l > 2}
                \frac{1}{k!l!}\frac{\partial^{k + l}f}{\partial x_1^k\partial x_2^l}(x_1, x_2)X_1^k X_2^l.
        \end{align*}
        By the assumption on $f$, the sum in the previous line contains finitely many elements: all mixed derivatives with $k+l>n$ are zeros. Then, the assumption
        $\E|X|^n$ allows us to take mathematical expectations. As a result we obtain 
        \begin{align*}
                \E\bigl[f(x_1 + X_1, &x_2 + X_2) - f(x_1, x_2)\bigr] = \E\Bigl[\frac{\partial f}{\partial x_1}
                (x_1, x_2) X_1 + \frac{\partial f}{\partial x_2}(x_1, x_2) X_2 \\
                &+ \frac{1}{2} \left(\frac{\partial^2 f}{\partial x_1^2}(x_1, x_2) X_1^2 + 2 
                \frac{\partial^2 f}{\partial x_1 \partial x_2}(x_1, x_2) X_1 X_2 + \frac{\partial^2 f}
                {\partial x_2^2}(x_1, x_2) X_2^2 \right) 
                \\
                &+ \sum\limits_{3\le k + l\le n}
                \frac{1}{k!l!}\frac{\partial^{k + l}f}{\partial x_1^k\partial x_2^l}(x_1, x_2)X_1^k X_2^l\Bigr].
        \end{align*}
        According to condition \eqref{eq:1_and_2-moments}, 
        \begin{align*}
                \E X_1 = \E X_2 = 0,\quad\E X_1^2 = \E X_2^2 = 1\text{ and }\E X_1 X_2 = 0.
        \end{align*}
        Thus
        \begin{align*}
        \E\left[\frac{\partial f}{\partial x_1}(x_1, x_2) X_1 
          + \frac{\partial f}{\partial x_2}(x_1, x_2) X_2\right] = 0 
        \end{align*}
        and 
        \begin{align*}
                \E\left[\frac{\partial^2 f}{\partial x_1^2}(x_1, x_2) X_1^2 + 2 \frac{\partial^2 f}{\partial 
                x_1 \partial x_2}(x_1, x_2) X_1 X_2 + \frac{\partial^2 f}{\partial x_2^2}(x_1, x_2) X_2^2 
                \right] = \Delta f(x_1,x_2).
        \end{align*}
        Consequently,
        \begin{align*}
        L_Xf(x)=\frac{1}{2}\Delta f(x)+   
        \sum\limits_{3\le k + l\le n}
                \frac{1}{k!l!}\frac{\partial^{k + l}f}{\partial x_1^k\partial x_2^l}(x_1, x_2)\E\left[X_1^k X_2^l\right]
        \end{align*}
        Noticing that
        \begin{align*}
                \frac{\partial^{k + l}f}{\partial x_1^k\partial x_2^l}(x_1, x_2)
        \end{align*}
        is polynomial of degree $n-3$ for all $k + l \ge 3$, we complete the proof of the lemma.
\end{proof}
\begin{corollary}\label{nabla_n}
Let the conditions of Lemma~\ref{taylor_exp} be valid. Consider the unique decomposition of the polynomial $f$ into homogeneous components:
$$
f(x)=\sum_{j=0}^n f_j(x),
$$
where every polynomial $f_j$ is homogeneous of degree $j$.\\
If $L_Xf(x)=0$ for all $x\in R\cap K_\alpha$ then $\Delta f_n\equiv0$.
\end{corollary}
\begin{proof}
        The claim is trivial if $f_n = 0.$ Assume now that $f_n \ne 0.$
        According to Lemma~\ref{lem:f=zero}, the polynomial $L_X f(x)=0$ for all $x\in\R^2$. Furthermore, by Lemma \ref{taylor_exp},
        \begin{align*}
                L_Xf(x) = \frac{1}{2}\Delta f(x) + g_f(x),
        \end{align*}
        where $g_f$ has degree not greater than $n - 3.$
        Using the equality $\Delta f = \sum\limits_{i = 0}^n \Delta f_i$ and noting that $\Delta f_j$ is homogeneous of degree $j-2$, we infer that
        $$
        L_Xf(x) = \frac{1}{2}\Delta f_n(x) + \bar{g}_f(x),
        $$
        where the degree of $\bar{g}_f$ does not exceed $n-3$. Combining this with the fact that $L_Xf(x)=0$ for all $x$, we conclude that $\Delta f_n(x)=0$
        for all $x$.
\end{proof}
\subsection{Reformulation of the problem in terms of linear equations.}
Let $\R_n[x_1, x_2]$ denote the linear space of homogeneous polynomials of degree $n$.
Let $\{v_{n,i}\}$ denote the standard basis in $\R_n[x_1, x_2]$, that is,
\begin{align*}
        v_{n, i} := x_2^i x_1^{n - i},\quad i = 0,1,\dots, n.
\end{align*}
We now consider $\frac{1}{2}\Delta$ as a linear operator from 
$\R_n[x_1, x_2]$ to $\R_{n-2}[x_1, x_2]$. To write down the corresponding matrix
we determine the images of all basis vectors under $\frac{1}{2}\Delta$:
\begin{align*}
        \frac{1}{2}\Delta v_{n, i} &= \Bigl(\frac{1}{2}\frac{\partial^2 
        (x_2^ix_1^{n - i})}{\partial x_2^2} + 
        \frac{1}{2}\frac{\partial^2 (x_2^ix_1^{n - i})}{\partial x_1^2}\Bigr) 
        \\
        &= \frac{i(i - 1)}{2}x_2^{(i - 2) \vee 0}x_1^{n - i} + \frac{(n - i)(n - i - 1)}{2}
        x_2^ix_1^{(n - i - 2)\vee 0} 
        \\
        &= \binom{i}{2}v_{n - 2, (i - 2)\vee 0} + \binom{n - i}{2}v_{n - 2,i\wedge(n-2)},
        \quad i=0,1,\ldots,n.
\end{align*}
Thus, the corresponding matrix $M_n$ is of size $(n + 1) \times (n - 1)$ 
and its $i$-th row has the form:
\begin{align*}
        \bigl(\underbrace{0,\:\:\: \dots,\:\:\: 0,}_{i - 1}\:\:\:
        \begin{matrix}
                \binom{n - i + 1}{2},& 0,& \binom{i + 1}{2}, 
        \end{matrix}\:\:\:
        \underbrace{0,\:\:\: \dots,\:\:\: 0}_{n - i - 1}\bigr).
\end{align*}
It is clear that all rows are linearly independent and thus the rank of this matrix is
$n - 1$. We now add to this matrix two more rows which will describe later the boundary conditions on $\partial K_\alpha$:
\begin{align}\label{ML_def}
        &\bigl(1,\:\:\:0,\:\:\:\dots,\:\:\: 0,\:\:\:\dots,\:\:\: 0\bigr),\\
        &\bigl(1,\:\:\: b,\:\:\:\dots,\:\:\: b^i,\:\:\:\dots,\:\:\: b^n\bigr),
\end{align}
where $b = \tan{\alpha}$.
As a result, we obtain a matrix $M_{n,b}$ of size $(n+1)\times(n+1)$.

\begin{example} Let us consider special cases $n=3$, $n=4$ with $\alpha=\frac{\pi}{4}$.
For this value of $\alpha$ we have $b = 1.$ Consequently, we get the matrices
        \begin{align*}
                M_{3, 1} = 
                \begin{pmatrix}
                        3 & 0 & 1 & 0\\
                        0 & 1 & 0 & 3\\
                        1 & 0 & 0 & 0\\
                        1 & 1 & 1 & 1
                \end{pmatrix}
                \quad\text{and}\quad
                M_{4, 1} = 
                \begin{pmatrix}
                        6 & 0 & 1 & 0 & 0\\
                        0 & 3 & 0 & 3 & 0\\
                        0 & 0 & 1 & 0 & 6\\
                        1 & 0 & 0 & 0 & 0\\
                        1 & 1 & 1 & 1 & 1
                \end{pmatrix}.
        \end{align*}
\end{example}
Our next goal is to find conditions which guarantee the existence of the unique solution to the following system of linear equations: 
\begin{align}\label{main_matrix_eq}
        M_{n, b} 
        \begin{pmatrix}
                a_{n, 0}\\
                \vdots\\
                a_{n, n - 2}\\
                a_{n, n - 1}\\
                a_{n, n}\\
        \end{pmatrix}
        =
        \begin{pmatrix}
                c_{n - 2, 0}\\
                \vdots\\
                c_{n - 2, n - 2}\\
                0\\
                0\\
        \end{pmatrix},
\end{align}
where $c_{n - 2, 0}, \dots, c_{n - 2, n - 2}$ and $b$ are arbitrary real numbers.
Our construction of the matrix $M_{n,b}$ allows one to reformulate \eqref{main_matrix_eq} as a boundary problem for the Laplace operator. Indeed, if we set 
$$
f(x_1,x_2):=\sum_{i=0}^n a_{n,i}v_{n,i}\in\R_n[x_1,x_2],\quad
g(x_1,x_2):=\sum_{i=0}^{n-2} c_{n-2,i}v_{n-2,i}\in\R_{n-2}[x_1,x_2]
$$
and $b=\tan\alpha$, then \eqref{main_matrix_eq} is equivalent to
$$
\frac{1}{2}\Delta f = g
\quad\text{with boundary conditions}\quad
f(t,0)=f(t\cos\alpha,t\sin\alpha)=0,\ t\in\R.
$$
\begin{proposition}\label{mx_uq}
Assume that $n\ge 3$ and that $b\ne\tan{\frac{q\pi}{n}}$ for all $q\in\N$.
Then the system \eqref{main_matrix_eq} possesses unique solution for any real numbers  $c_{n - 2, 0}, \dots, c_{n - 2, n - 2}$.
Furthermore, if $b=\tan{\frac{q\pi}{n}}$ with some $q\in\N$ then the kernel of $M_{n,b}$ is $1$-dimensional.
\end{proposition}
\begin{proof}
        It is immediate from the structure of the first $(n-1)$ rows in the matrix
        $M_{n,b}$ that one can separate variables $a_{n,i}$ with even and odd indices $i$.
        
        We first show that the variables  $a_{n ,2k}$, $k = 0,\dots,\bigl[\frac{n}{2}\bigr]$ can be uniquely determined by using a recurrent formula. Indeed, 
        the pre-last equation implies that $a_{n, 0} = 0$. Furthermore, for every 
        $k = 1,\dots,\bigl[\frac{n}{2}\bigr]$ we have the following recurrent relation:
        \begin{align*}
                \binom{n - 2k + 2}{2}a_{n, 2k - 2} + \binom{2k}{2}a_{n, 2k} = c_{n, 2k - 2},
        \end{align*}
        which is equivalent to
        \begin{align*}
                a_{n, 2k} = \frac{c_{n, 2k - 2} - \binom{n - 2k + 2}{2}a_{n, 2k - 2}}
                {\binom{2k}{2}}.
        \end{align*}
        This recursion allows us to uniquely determine all numbers $a_{n,i}$ with even indices $i$.

        We now turn to variables with odd indices. Set 
        $$
        n_{odd} = \biggl\lceil \frac{n - 3}{2}\biggr\rceil
        \quad\text{and}\quad
        r_{n - 2} := -\sum\limits_{k = 0}^{[\frac{n}{2}]}b^{2k}a_{n, 2k}.
        $$ 
        Then, the system of equations on $\{a_{2k + 1}\}$ can be written in the matrix form
        \begin{align*}
                M^{odd}_{n, b}
                \begin{pmatrix}
                        a_{n, 1}\\
                        a_{n, 3}\\
                        \vdots\\
                        a_{n, 2k - 1}\\
                        \vdots\\
                        a_{n, 2n_{odd} - 1}\\
                        a_{n, 2n_{odd} + 1}\\
                \end{pmatrix}
                =
                \begin{pmatrix}
                        c_{n - 2, 1}\\
                        c_{n - 2, 3}\\
                        \vdots\\
                        c_{n - 2, 2k - 1}\\
                        \vdots\\
                        c_{n - 2, 2n_{odd} - 1}\\
                        r_{n - 2}\\
                \end{pmatrix}&,
        \end{align*}
        where the $i$-th row with $i\le n_{odd}$ equals
        \begin{align*}
                \bigl(\underbrace{0,\:\:\: \dots,\:\:\: 0,}_{i - 1}\:\:\:
                \begin{matrix}
                        \binom{n - 2i + 1}{2},& \binom{2i + 1}{2},
                \end{matrix}\:\:\:
                \underbrace{0,\:\:\: \dots,\:\:\: 0}_{n_{odd} - i}\bigr)
        \end{align*}
        and the last row is given by
        \begin{align*}
                \bigl(b,\:\:\:\dots,\:\:\: b^{2i + 1},\:\:\:\dots,\:\:\: b^{2n_{odd} + 1}\bigr).
        \end{align*}
We now simplify this system by applying so-called elementary row-addition matrices.
Let us recall the corresponding definition.
\begin{definition}
                Let $E_n$ denote $n$-dimensional identity matrix and let also 
                $\{e_{n, i}\}$ be the standard basis in $\R^n$. Matrix of the form
                \begin{align*}
                        E_{n}(i, j, \lambda) := E_n + \lambda e_{n, i}e_{n, j}^T
                \end{align*}
                is called elementary row-addition matrix of dimension $n$ with parameters $i, j$ and $\lambda$. 
                Multiplying by this matrix on the left leads to a matrix where the 
                $j$-th row has been multiplied  by $\lambda$ and added to $i$-th row.
        \end{definition}
We shall use such row-addition matrices to transform the matrix $M_{n,b}$ to 
an upper-triangular one.
\begin{lemma}\label{mx_simpl_lemma}
                 Set \begin{align*}
                        \lambda_{n, 1} = (-1)\cdot\frac{b}{\binom{n - 1}{2}}
                        \ \text{ and }\ 
                                        &\lambda_{n, k + 1} =  (-1)\cdot\frac{b^{2k + 1} + \lambda_{n, k}
                                        \binom{2k + 1}{2}}{\binom{n - 2k - 1}{2}},\:\: 2\le k  \le n_{odd}-1.
                                \end{align*}
                                Then we have
                                \begin{align}\label{mx_simpl_1}
                                \nonumber
                                        UM_{n, b}^{odd} &:= \revprod\limits_{k = 1}^{n_{odd}}
                                        E_{n_{odd} + 1}(n_{odd} + 1, k, \lambda_{n, k})\cdot M_{L, n, b}^{odd}
                                        \\
                                        &=
                                        \begin{pmatrix}
                                                \binom{n - 1}{2} & \binom{3}{2} & 0 && \cdots & 0\\ 
                                                &\ddots & \ddots &&&\\
                                                \vdots&&\binom{n - 2k - 1}{2} &  \binom{2k + 1}{2} &&\vdots\\ 
                                                &&&\ddots & \ddots &\\
                                                &&&&\binom{n + 1 - 2n_{odd}}{2} & \binom{2n_{odd} + 1}{2}\\
                                                0 && \cdots && 0 & C(n)\cdot u_{\pi/n}(1, b)\\
                                        \end{pmatrix},
                                \end{align}
                                where $C(n)$ depends only on $n$ and $\revprod$ denotes the left product of matrices, that is,
                                $$
                                \revprod_{k = 1}^l A(k) := A(l)\revprod_{k = 1}^{l - 1}A(k),
                                $$
                                for an arbitrary family of matrices $\{A(k)\}.$
        \end{lemma}
        \begin{proof}
         We have to show that the multiplication by matrices $E_{n_{odd+1}}(n_{odd}+1,k,\lambda_{n,k})$ eliminates all the  off-diagonal elements of the last row.  By the definition, multiplication from the left by $E_{n_{odd+1}}(n_{odd}+1,k,\lambda_{n,k})$ is equivalent to adding $k$th row multiplied by $\lambda_{n,k}$ to the last row. It is immediate from the definition of $M_{n,b}^{odd}$ that the $k$th operation changes $k$th and $(k+1)$th elements of the last row. Let $\theta_k$ denote the $(k+1)$th element of the last row after $k$ operations. Clearly,
         $$
         \theta_k=b^{2k+1}+\lambda_{n,k} {2k+1\choose 2}
         $$
         and, consequently, the value of the $k$th element after $k$ operations is equal to $\theta_{k-1}+\lambda_{n,k}{n-2k+1\choose 2}$. Recalling the definition of $\lambda_{n,k}$, we see that this value is zero. Thus, it remains to compute $\theta_{n_{odd}}$, which is the last element of the last row after completing all operations. To this end, we prove by induction that 
                        \begin{align*}
                                \lambda_{n, k} = \frac{\sum\limits_{l = 0}^{k - 1}(-1)^{l + 1}
                                b^{2(k - 1 - l) + 1}\prod\limits_{i = 0}^{k - l - 2}\binom{n - 1 - 2i}{2}
                                \prod\limits_{j = 0}^{l - 1}\binom{2k - 1 - 2j}{2}}
                                {\prod\limits_{s = 0}^{k - 1}\binom{n - 1 - 2s}{2}}. 
                        \end{align*}
                        It is clear that this representation is true for 
                        $k=1$. So, the basis case of the induction is checked.
                        Assume now that the claim holds for some $k\ge 1$. Then, by the definition of $\lambda_{n,k+1}$,
                        \begin{align*} 
                                \lambda_{n, k + 1} &=  -\frac{b^{2k + 1} + \lambda_{n, k} 
                                \binom{2k + 1}{2}}{\binom{n - 2k - 1}{2}}
                                \\
                                &=-\Bigg[\frac{b^{2k + 1}}{\binom{n - 2k - 1}{2}} +
                                \frac{\sum\limits_{l = 0}^{k - 1}(-1)^{l + 1}
                                b^{2(k - 1 - l) + 1}\prod\limits_{i = 0}^{k - l - 2}\binom{n - 1 - 2i}{2}
                                \prod\limits_{j = 0}^{l - 1}\binom{2k - 1 - 2j}{2}}
                                {\prod\limits_{i = 0}^{k - 1}\binom{n - 1 - 2i}{2}}\cdot
                                \frac{\binom{2k + 1}{2}}{\binom{n - 2k - 1}{2}}\Bigg]
                                \\
                                &=\Biggl[-b^{2k + 1}\cdot\prod\limits_{s = 0}^{k - 1}\binom{n - 1 - 2s}{2} + 
                                \sum\limits_{l = 0}^{k - 1}(-1)^{l + 2}
                                b^{2(k - 1 - l) + 1}\prod\limits_{i = 0}^{k - l - 2}\binom{n - 1 - 2i}{2}
                                \\
                                &\hspace{2cm}\times\prod\limits_{j = 0}^{l - 1}\binom{2k - 1 - 2j}{2}
                                \binom{2k + 1}{2}\Biggr]
                                \frac{1}{\prod\limits_{s = 0}^{k - 1}\binom{n - 1 - 2s}{2}\cdot
                                \binom{n - 2k - 1}{2}}.
                        \end{align*}
                    Shifting the indices $l$ and $j$, one obtains 
                    \begin{align*}
                        &\sum\limits_{l = 0}^{k - 1}(-1)^{l + 2}
                                b^{2(k - 1 - l) + 1}\prod\limits_{i = 0}^{k - l - 2}\binom{n - 1 - 2i}{2}
                                \prod\limits_{j = 0}^{l - 1}\binom{2k - 1 - 2j}{2} \binom{2k + 1}{2}\\
                        &=\sum\limits_{l = 1}^{k}(-1)^{l + 1}
                                b^{2((k+1) - 1 - l) + 1}\prod\limits_{i = 0}^{(k+1) - l - 2}\binom{n - 1 - 2i}{2}
                                \prod\limits_{j = 0}^{l - 1}\binom{2(+1)k - 1 - 2j}{2}.      
                    \end{align*}
                    Having this equality and noticing that 
                    $$
                    -b^{2k+1}\prod_{i=0}^{k-1}{n-1-2i\choose 2}
                    =(-1)^1b^{2(k+1)-1}\prod_{i=0}^{(k+1)-2}{n-1-2i\choose 2},
                    $$
                    we complete the proof of the induction step.
                    Consequently,
                        \begin{align*}
                                &\theta_{n_{odd}} = b^{2n_{odd} + 1} + \lambda_{n, n_{odd}}\binom{2n_{odd} + 1}{2} 
                                \\
                                &= b^{2n_{odd} + 1} \\
                                &\hspace{1cm}+\frac{\sum\limits_{l = 0}^{n_{odd} - 1}(-1)^{l + 1}
                                b^{2(n_{odd} - 1 - l) + 1}\prod\limits_{i = 0}^{n_{odd} - l - 2}
                                \binom{n - 1 - 2i}{2}\prod\limits_{j = 0}^{l - 1}\binom{2n_{odd} - 1 - 2j}{2}}
                                {\prod\limits_{s = 0}^{n_{odd} - 1}\binom{n - 1 - 2s}{2}}
                                \binom{2n_{odd} + 1}{2}.
                        \end{align*}
                        Shifting the indices $l$ and $j$ once again, we obtain
                        \begin{align*}
                                \theta_{n_{odd}}= 
                                \frac{\sum\limits_{l = 0}^{n_{odd}}(-1)^{l}
                                b^{2(n_{odd} - l) + 1}\prod\limits_{i = 0}^{n_{odd} - l - 1}
                                \binom{n - 1 - 2i}{2}\prod\limits_{j = 0}^{l - 1}\binom{2n_{odd} + 1 - 2j}{2}}
                                {\prod\limits_{s = 0}^{n_{odd} - 1}\binom{n - 1 - 2s}{2}}.
                        \end{align*}
                        Noticing that
                        \begin{align*}
                                \prod\limits_{t = 0}^N\binom{M - 2t}{2} = \frac{M!}{(M - 2N - 2)!}
                                2^{-N - 1},
                        \end{align*}
                        for all $M, N\ge0$, we can further simplify the expression above:
                        \begin{align*}
                                \theta_{n_{odd}} 
                                =\frac{1}{{n\choose 2n_{odd}+1}}
                                \sum\limits_{l = 0}^{n_{odd}}(-1)^{l}b^{2n_{odd} + 1 - 2l}\binom{n}{2n_{odd} + 1 - 2l}.
                        \end{align*}
                        Recalling that $u_{\pi/n}(x_1,x_2)=\Im(x_1+ix_2)^n$, we infer that 
                        $$
                        \theta_{n_{odd}} 
                                =\frac{1}{{n\choose 2n_{odd}+1}}u_{\pi/n}(1,b).
                        $$
                This finishes the proof of the lemma.
        \end{proof}        
    As we have already seen, all $a_{n,j}$ with even indices $j$ are uniquely determined by the numbers $c_{n,j}$. Now, Lemma~\ref{mx_simpl_lemma} implies that the matrix $M_{n,b}^{odd}$ is invertible for all $b$ such that $u_{\pi/n}(1,b)$ is not zero. But it is immediate from the definition that all zeros of $u_{\pi/n}(1,b)$ are of the form $b=\tan{(q\pi/n)}$, $q\in\N$.  Thus, the first claim is proved.  The fact that the kernel $M_{n,b}$ is one-dimensional for $b=\tan{(q\pi/n)}$ also follows from Lemma~\ref{mx_simpl_lemma}. Thus, the proof of the proposition is complete.
\end{proof}
The following claim is immediate from the fact that the kernel of $M_{n,\tan(\pi/n)}$ is one-dimensional.
\begin{corollary}\label{nabla_ker}
        Let $f$ be a homogeneous polynomial of degree $n$.
        If $\Delta f = 0$ and $f\vert_{\partial K_{\pi/n}}=0$
        then $f = c u_{\pi/n}$ for some constant $c\in\R.$
\end{corollary}
\subsection{Proof of Theorem~\ref{thm:harmonic}.} 
Let $h$ be a polynomial of degree $m$. Using bases $\{v_{k,j}\}$ we can represent $h$
as follows
$$
h(x_1,x_2)=\sum_{k=0}^m h_k(x_1,x_2),
$$
where 
$$
h_k(x_1,x_2):=\sum_{j=0}^k a_{k,j}v_{k,j}
\quad k=0,1,\ldots,m.
$$
Let us notice that the condition 
$$
h(t,0)=h(t\cos(\pi/m),t\sin(\pi/m))=0
\quad\text{for all }t
$$
is equivalent to
$$
h_k(t,0)=h_k(t\cos(\pi/m),t\sin(\pi/m))=0
\quad\text{for all $t$ and all $k$}.
$$

To construct the desired polynomial and to show the uniqueness we shall use an 
"inverse" induction and prove the following statement.
\emph{For every $l \le m$ there exist unique, up to a constant multiplier, numbers
$\{a_{i, j}\}$ with $l\le i\le m$ and 
        $0\le j \le i$ such that
        \begin{itemize}
        \item[(a)] $\deg L_X\left(\sum\limits_{k = l}^{m}h_k\right) \le l - 3$,
        \item[(b)] $h_k(t, 0) = h_k\bigl(t\cos(\pi/m), t\sin(\pi/m)\bigr) = 0$
                   for all $t$ and all $l\le k\le m$.
        \end{itemize}
    }
Let us first check the basis case $l=m$. According to Lemma~\ref{lem:BM1},
the polynomial $Q_m(x)=u_{\pi/m}(x)$ possesses all required properties. Let $h_m(x)$
be any $m$-homogeneous polynomial which satisfies (a) and (b). Combining (a) with Lemma~\ref{taylor_exp}, we infer that $\Delta h_m(x)=0$.  Since $h_m$ satisfies the boundary conditions in (b), we may apply Corollary \ref{nabla_n} to conclude that $h_m=cQ_m$.

Assume now that the statement holds for some $l+1\le m$. Then there exist numbers
$\{c_{i,j}\}$ such that 
$$
L_X\left(\sum_{k=l+1}^m h_k\right)
=\sum_{i=0}^{l-2}\sum_{j=0}^i c_{i,j}v_{i,j}.
$$
According to Proposition \ref{mx_uq}, the system of linear equations
\begin{align*}
                M_{l, b}  
                \begin{pmatrix}
                        a_{l, 0}\\
                        \vdots\\
                        a_{l, l - 2}\\
                        a_{l, l - 1}\\
                        a_{l, l}\\
                \end{pmatrix}
                =
                \begin{pmatrix}
                        -c_{l - 2, 0}\\
                        \vdots\\
                        -c_{l - 2, l - 2}\\
                        0\\
                        0\\
                \end{pmatrix}
        \end{align*}
has unique solution. In other words, there exists unique $l$-homogeneous polynomial
$h_l$ such that 
$$
\frac{1}{2}\Delta h_l=-\sum_{i=0}^{l-2}c_{l-2,i}v_{l-2,i}
$$
and 
$$
h_l(t, 0) = h_l\bigl(t\cos(\pi/m), t\sin(\pi/m)\bigr) = 0
\quad\text{for all }t.
$$
Combining this with Lemma~\ref{taylor_exp}, we infer that 
the degree of $L_X\left(\sum_{k=l}^m h_k\right)$ does not exceed $l-3$.
Thus, the step of the induction is proved.

The claim we have shown above implies the existence and the uniqueness (up to a constant multiplier) of a polynomial $h(x)$ of degree $m$ such that 
$$
L_Xh(x)=0\text{ for all }x\in\R^2
\text{ and }h\vert_{\partial K_{\pi/m}}=0.
$$
Furthermore, the $m$-homogeneous part of $h$ is proportional to $u_{\pi/m}$. 

If the condition \eqref{eq:no-overshoot} holds then, for every $x\in R\cap K_{\pi/m}$, one has
$$
\{x + X\notin K_{\pi/m}\}
=\{x + X\in\partial K_{\pi/m}\}.
$$
Consequently, for every $x\in R\cap K_{\pi/m}$,
        \begin{align*}
            \E\bigl[h(x + X), x + X\in K_{\pi/m}\bigr] 
                &= \E\bigl[h(x + X)\bigr] - 
                \E\bigl[h(x + X),x + X\notin K_{\pi/m}\bigr] \\
                &= \E\bigl[h(x + X)\bigr] - 
                \E\bigl[h(x + X), x + X\in\partial K_{\pi/m}\bigr]\\
                &=\E\bigl[h(x + X)\bigr],
        \end{align*}
        in the last step we have used the equality $h\vert_{\partial K_{\pi/m}}=0$.
        Noting that $L_Xh=0$ is equivalent to $\E\bigl[h(x + X)\bigr]=h(x)$, we conclude that $h$ is harmonic for $S(n)$ killed at leaving 
        $K_{\pi/m}$.  The fact that all coefficients can be expressed in terms of moments follows from Lemma~\ref{taylor_exp}. Thus, the first part of the theorem is proved.

Our approach allows us to prove the following result.
\begin{corollary}\label{low_deg_harmonic}
                Let $f$ be  a polynomial satisfying 
                \begin{align}\label{eq:f-cond}
                        L_Xf = 0\quad\text{and}\quad f\vert_{\partial K_{\pi/m}}=0.
                \end{align}
                If the degree of $f$ is smaller than $m$ then $f =0.$
        \end{corollary}
        \begin{proof}
            Let $r<m$ denote the degree of the polynomial $f$. As usual we decompose
            $f$ into homogeneous parts: $f=\sum_{j=0}^r f_r$. 
            According to Corollary \ref{nabla_n}, $\Delta f_r = 0$. 
            Moreover, $f_r\vert_{\partial K_{\pi/m}} = 0$.
            We can rewrite these equalities in terms of linear equations: 
                \begin{align*}
                        M_{r, b}  
                        \begin{pmatrix}
                                a_{r, 0}\\
                                \vdots\\
                                a_{r, r - 2}\\
                                a_{r, r - 1}\\
                                a_{r, r}\\
                        \end{pmatrix}
                        =
                        \begin{pmatrix}
                                0\\
                                \vdots\\
                                0\\
                                0\\
                                0\\
                        \end{pmatrix},
                \end{align*}
                where $a_{r,j}$ are the coefficients of $f_r$ in the basis $v_{r,j}$. 
                The assumption $r<m$ implies that $b=\tan(\pi/m)\neq \tan(q\pi/r)$ for all $q\in\N$. Therefore, applying Proposition \ref{mx_uq}, we infer that $a_{r,j}=0$ for all $j$ is the only solution to the system above. In other words, $f_r=0$. This implies that $f-f_r$ satisfies all the conditions in \eqref{eq:f-cond}. Repeating the same arguments, we conclude that all components $f_j$ are equal to zero.
        \end{proof}

Let us now turn to the proof of Theorem~\ref{thm:harmonic}(b).    

Let $h$ be a positive harmonic polynomial of degree $m$ for $S(n)$ killed at leaving $K_\alpha$, that is,
$$
\E[h(x+X);x+X\in K_\alpha]=h(x)
\quad\text{and}\quad 
h(x)\ge 0
\quad\text{for all }x\in R\cap K_\alpha.
$$
The boundary condition $h\vert_{R\cap \partial K_\alpha}=0$ and the assumption \eqref{eq:no-overshoot} imply that 
$$
L_X h(x)=0
\quad\text{for all }x\in R\cap K_\alpha.
$$
Then, applying Corollary~\ref{nabla_n}, we conclude that 
$$
\Delta h_m(x)=0,\ x\in\R^2
\quad\text{and}\quad 
h_m\vert_{\partial K_\alpha}=0
$$
where $h_m$ denotes the homogeneous part of $h$ of degree $m$. 
Writing this boundary problem in terms of linear equations, we have 
\begin{align}\label{eq:harm-equ}
                M_{m, b} 
                \begin{pmatrix}
                        a_{m, 0}\\
                        \vdots\\
                        a_{m, m - 2}\\
                        a_{m, m - 1}\\
                        a_{m, m}\\
                \end{pmatrix}
                =
                \begin{pmatrix}
                        0\\
                        \vdots\\
                        0\\
                        0\\
                        0\\
                \end{pmatrix},
        \end{align}
        where $b = \tan{\alpha}$ and $\{a_{m,i}\}$ are coefficients of $h_m$ in the standard basis $\{v_{m,ik}\}$, that is, 
        $h_m(x)=\sum_{i=0}^ma_{m,i}v_{m,i}$.

        Proposition~\ref{mx_uq} implies that if the system \eqref{eq:harm-equ} possesses a non-trivial solution then $\alpha=q\pi/m$ for some $q\in\{1,2,\ldots,m-1\}$. Thus, it remains to show that $q=1$. Assume that this is not true, that is, $q\ge2$.  As we have shown in Proposition~\ref{mx_uq}, the kernel of the matrix $M_{m,\tan(q\pi/m)}$ is one-dimensional. Noticing that $u_{\pi/m}$ is a solution to \eqref{eq:harm-equ}, we infer that all 
        non-trivial solutions have the form $cu_{\pi/m}$ and, consequently, $h_m(x)=c_0u_{\pi/m}$ with some positive $c_0$. If $q\ge2$ then the function
        $u_{\pi/m}(x_1,x_2)$ is negative for all $x_1,x_2$ such that ${\rm arg}(x_1+ix_2)\in(\pi/m,2\pi/m)$. But this contradicts the assumption that $h_m$
        is positive. Thus, $q=1$ and the proof is complete.
\section{Proof of Theorem~\ref{thm:uniqueness}}
We first prepare the proof of the theorem by deriving some properties of random walks killed at leaving $K_\alpha$.
\begin{lemma}\label{boundary_harmonic}
        Assume that \eqref{eq:1_and_2-moments} and \eqref{eq:no-overshoot} hold. Then,
        for every $x\in K_\alpha\cap R$, one of the following properties is true:
        \begin{itemize}
                \item $\P\bigl(x + X\in\partial K_\alpha\vert X\neq0\bigr) = 1,$
                \item for every $M>0$ there exists $n_0=n_0(M)$ such that 
                        \begin{align*}
                        \P\bigl(\max\{x_1 + S_1(n_0), x_2 + S_2(n_0)\} > M , \tau_x > n_0\bigr) > 0.
                        \end{align*}
        \end{itemize}
\end{lemma}
\begin{proof}
        Let $v_1$ and $v_2$ be the generators of $R$, that is, 
        $$
        R = \{\lambda_1 v_1 + \lambda v_2,\:\:\lambda_1,\lambda_2\in\Z\}.
        $$ 
        The assumption \eqref{eq:no-overshoot} implies that we can take  $v_1,v_2\in\partial K_\alpha$. 
        Thus there exists a linear transformation $\phi$ such that
        \begin{align*}
                \phi(v_1) = (1, 0),\quad\phi(v_2) = (0, 1)
                \quad\text{and}\quad\phi(K_\alpha) = \R^2_+.
        \end{align*}
        After this transformation $\phi(R) = \Z^2$ and $\phi(K_\alpha\cap R) = \Z^2_+.$
        Set $X^\phi := \phi(X)$ and $S^\phi(k) = \phi(S(k))$ for every $k\in\N_0$. Combining \eqref{eq:1_and_2-moments} with the linearity of $\phi$, we conclude that
        \begin{align*}
               \E[X^\phi_1] = \E[X^\phi_2] = 0.
        \end{align*}
        Moreover, \eqref{eq:no-overshoot} implies that the walk $S^\phi$ satisfies
        \begin{align}\label{eq:no-overshoot-phi}
        \P(x+S^\phi(\tau_x^\phi)\in\partial \R_+^2)=1\quad\text{for all }x\in\Z_+^2,
        \end{align}
        where 
        $$
        \tau^\phi_x:=\inf\{n\ge 1:x+S^\phi(n)\notin\R_+^2\}.
        $$

        Assume first that there exist $a, b \ge 0$ such that $a + b > 0$ and 
        \begin{align*}
                \P\bigl(X^\phi = (a, b)\bigr) > 0.
        \end{align*}
        Then there exists $n_0$ such that
        \begin{align*}
                \P\bigl(\max\{x_1 + S_1(n_0), x_2 + S_2(n_0)\} > M, \tau_x > n_0\bigr) > 0,\quad x\in K_\alpha\cap R;
        \end{align*}
        it suffices to have a large number of steps $\phi^{-1}((a,b))$ at the beginning.

        Thus, it remains to consider the case when $\P\bigl(X^\phi = (a, b)\bigr)=0$ for all $a,b\ge0$ with $a+b>0$. In other words, all non-zero jumps of $X^\phi$ are of the form  $(a,-1)$ and $(-1,b)$ with $a,b\ge-1$. (Notice that smaller values of $a,b$ are impossible because of \eqref{eq:no-overshoot-phi}.) We next notice that at least one of the values of $a,b$ should be strictly bigger than one. Indeed, if $a,b\le 1$ then all non-zero jumps are in the half-plane
        $\{(x_1,x_2):\ x_1+x_2\le0\}$. But this contradicts the assumption \eqref{eq:1_and_2-moments}.
        
        Assume, without loss of generality, that $\P(X^\phi=(a,-1))>0$ for some $a>1$.
        Moreover, $\E X_1^\phi=\E X_2^\phi=0$  implies the existence of $b\ge1$ such that $\P(X^\phi=(-1,b))>0$. If the starting point $x=(x_1,x_2)\in\Z_+^2$ is such $x_2>1$ then we can alternate jumps $\phi^{-1}((a,-1))$ and $\phi^{-1}((-1,b))$ to conclude that $\P(x_1+S_1(n_0)>M,\tau_x>n_0)$ is positive for every sufficiently large $n_0$. If $x=(x_1,1)$ with some $x_1>1$ then we make first the step $\phi^{-1}((-1,b))$ reducing the problem to a starting point with $x_2>1$. Thus, it remains to consider the starting point $x=(1,1)$. But it is clear that $\P((1,1)+X^\phi\in\partial R_+^2|X\neq0)=1$. This completes the proof of the lemma. 
\end{proof}
\begin{remark}
\label{rem:V-pos}
In the proof of the previous lemma we have shown that if \eqref{eq:no-overshoot} holds
then there exists $n_0$ such that 
$$
\P\left(\max\{x_1+S_1(n_0),x_2+S_2(n_0)\}>M,\tau_x>n_0\right)>0
$$
for all $x\neq \phi^{-1}(\widehat{x})$, where $\widehat{x}=\phi^{-1}((1,1))$. Furthermore, the same property holds also for $\widehat{x}$ in the case when $X^\phi$ has a non-zero jump in $[0,\infty)^2$. Combining this with Lemma~4 in \cite{EW25}, we obtain for the non-negative harmonic function $V$ the following lower bound
\begin{align*}
V(x)&=\E[V(x+S(n_0));\tau_x>n_0]\\
&\ge \E[V(x+S(n_0));\max\{x_1+S_1(n_0),x_2+S_2(n_0)\}>M\tau_x>n_0]\\
&\le \delta \P\left(\max\{x_1+S_1(n_0),x_2+S_2(n_0)\}>M,\tau_x>n_0\right)>0.
\end{align*}
Thus, the function $V$ is strictly positive on
$(K_{\pi/m}\cap R)\setminus\{\widehat{x}\}$. Furthermore, $V(\widehat{x})>0$ if $X^\phi$ has a non-zero jump in $[0,\infty)^2$.

Finally, if all non-zero jumps of $X^\phi$ are of the form $(a,-1)$ and $(-1,b)$ then 
$$
\P(\tau_{\widehat{x}}>n)=(\P(X=0))^n,\quad n\ge1.
$$
The exponential decay of this probability implies that $V(\widehat{x})=0$.
\hfill$\diamond$
\end{remark}
In the next lemma we describe the asymptotic behaviour of the polynomial $h$ constructed in Theorem~\ref{thm:harmonic}(a).
\begin{lemma}\label{h_assympt}
Let $h$ be the polynomial constructed in Theorem~\ref{thm:harmonic}. Then the following three statements hold true.
        \begin{itemize}
                \item[(a)] There exists a constant $C$ such that 
                        \begin{align*}
                                \left|h(x)\right| < C \delta(x)|x|^{m - 1},\quad x\in K_{\pi/m}.
                        \end{align*}
                \item[(b)] There exists $M>0$ such that   
                \begin{align*}
                                h(x) >\frac{1}{2} \delta(x)|x|^{m - 1}
                                \quad\text{for all } x\in K_{\pi/m}
                                \text{ with }|x|\ge M.
                        \end{align*}
                \item[(c)]One has
                        \begin{align*}
                                \lim\limits_{\delta(x)\to\infty} \frac{h(x) - u_{\pi/m}(x)}{|x|^{m - 1}\delta(x)}
                                =0 \quad\text{and}\quad\lim\limits_{\delta(x)\to\infty}\frac{h(x)}{u_{\pi/m}(x)} 
                                = 1.
                        \end{align*}
        \end{itemize}
\end{lemma}
\begin{proof}
        By the symmetry, it suffices to prove all three statements for $x=(x_1,x_2)\in K_{\pi/m}$ with ${\rm arg}(x_1+ix_2)\le \frac{\pi}{2m}$. For such points $x$ one has $\delta(x)=x_2$.

        According to Theorem~\ref{thm:harmonic}(a), $h(x)=u_{\pi/m}(x)+r_m(x)$ and the degree of $r_m$ does not exceed $m-1$. Boundary conditions $h(x_1,0)=u_{\pi/m}(x_1,0)=0$ imply that the polynomials $u_{\pi/m}$ and $r_m$ can be divided by $x_2$. Therefore, $h(x)=x_2(\widetilde{u}_m(x)+\widetilde{r}_m(x))$, where 
        $\widetilde{u}_{\pi/m}(x)=\frac{u_{\pi/m}(x)}{x_2}$ is homogeneous of degree $m-1$ and $\widetilde{r}_m(x)$ is polynomial of degree $m-2$. 

        The upper bound for $|h|$ is now quite obvious. Indeed, by the triangle inequality,
        $$
        \left|h(x)\right|\le x_2\left(\left|\widetilde{u}_{\pi/m}(x)\right|+|\widetilde{u}_m(x)|\right)\le Cx_2|x|^{m-1}.
        $$

        To obtain the lower bound in (b) we have to estimate from below the difference
        $\widetilde{u}_{\pi/m}(x)-\widetilde{r}_{m}(x)$. Since the degree of $\widetilde{r}_{m}(x)$ is $m-2$, we have $|\widetilde{r}_m(x)|\le C_1|x|^{m-2}$
        for some constant $C_1$.
        Furthermore, using the polar coordinates, we obtain 
        $$
        \widetilde{u}_{\pi/m}(x)=\frac{|x|^{m}\sin(m\beta)}{|x|\sin\beta},
        $$
        where $\beta={\rm arg}(x_1+ix_2)$.
        We have shown in the proof of Lemma~\ref{lem:BM1} that
        $\frac{\sin(m\beta)}{\sin\beta}\ge 1$ for all $\beta\in[0,\pi/(2m)]$. This implies that 
        $$
        \widetilde{u}_{\pi/m}(x)\ge |x|^{m-1}\quad \text{for all}\quad 
        x\in K_{\pi/m}\quad \text{with}\quad  
        {\rm arg}(x_1+ix_2)\le \frac{\pi}{2m}.
        $$
        Consequently,
        $$
        \widetilde{u}_{\pi/m}(x)-\widetilde{r}_{m}(x)\ge |x|^{m-1}-C_1|x|^{m-2}
        $$
        for all $x\in K_{\pi/m}$ with ${\rm arg}(x_1+ix_2)\le \frac{\pi}{2m}$.
        Thus, the bound in (b) holds if we take $M\ge 2C_1$.

        We notice next that 
        $$
        \frac{h(x)-u_{\pi/m}(x)}{x_2|x|^{m-1}}=\frac{\widetilde{r}_m(x)}{|x|^{m-1}}.
        $$
        Since the degree of $\widetilde{r}_m$ is $m-2$, the right hand side in the previous inequality goes to zero as $\delta(x)\to\infty$. This gives the first relation in (c). To prove the second limit in (c) we write
        $$
        \frac{h(x)}{u_{\pi/m}(x)}=1+\frac{h(x)-u_{\pi/m}(x)}{u_{\pi/m}(x)}
        =1+\frac{h(x)-u_{\pi/m}(x)}{\delta(x)|x|^{m-1}}
        \frac{\delta(x)|x|^{m-1}}{u_{\pi/m}(x)}.
        $$
        Applying now the lower bound in \eqref{eq:BM2}, we infer that 
        $$
        \left|\frac{h(x)}{u_{\pi/m}(x)}-1\right|\le \frac{h(x)-u_{\pi/m}(x)}{\delta(x)|x|^{m-1}}.
        $$
        We have already proved that the ratio on the right hand side converges to zero.
        Thus, the second equality is proved as well.
\end{proof}

Now we are in position to prove Theorem~\ref{thm:uniqueness}.
\begin{proof}[Proof of Theorem~\ref{thm:uniqueness}]
Assume first that $x\in K_{\pi/m}\cap R$ is such that $V(x)>0$.

Let $\widehat{\P}_x$ denote the Doob $h$-transform of the distribution of the walk $S(n)$ killed at leaving $K_{\pi/m}$. Under this new measure, $S(n)$ is a Markov chain with the transition kernel
$$
\widehat{P}(x,y)=\frac{V(y)}{V(x)}\P(x+S(1)=y,\tau_x>1),
\quad x,y\in K_{\pi/m}\cap R.
$$
Consider the sequence of random variables
$$
W_x(n):=\frac{h(x+S(n))}{V(x+S(n))},\quad n\ge0.
$$
Combining Lemma~\ref{h_assympt} and Lemma~4 in \cite{EW25}, we infer that these random variables are bounded by some constant $C$. Furthermore, by the definition of $\widehat{\P}_x$,
\begin{align*}
\widehat{\E}_x\left[W_x(n+1)|x+S(n)=y\right]
&=\sum_{z\in K_{\pi/m}\cap R}\frac{h(z)}{V(z)}\widehat{\P}(y,z)\\
&=\sum_{z\in K_{\pi/m}\cap R}\frac{h(z)}{V(z)}\frac{V(z)}{v(y)}\P(y+X=z)\\
&=\frac{1}{V(y)}\E[h(y+X);y+X\in K_{\pi/m}].
\end{align*}
Recalling that $h$ is harmonic, we conclude that 
$$
\widehat{\E}_x\left[W_x(n+1)|x+S(n)=y\right]=\frac{h(y)}{V(y)}.
$$
Therefore, the sequence $W_x(n)$ is a martingale under $\widehat{\P}_x$ and, in particular,
$$
\widehat{\E}_x[W_x(n)]=W_x(0)=\frac{h(x)}{V(x)},\quad n\ge0.
$$
We next compute the limit of the sequence $\widehat{\E}[W_x(n)]$. We start by splitting this expectation into two parts:
$$
\widehat{\E}_x[W_x(n)]=
\widehat{\E}_x[W_x(n);\delta(x+S(n))\le \log^2 n]
+\widehat{\E}_x[W_x(n);\delta(x+S(n))> \log^2 n].
$$
Recalling that $W_x(n)$ is bounded, we get 
$$
\widehat{\E}_x[W_x(n);\delta(x+S(n))\le \log^2 n]
\le C\widehat{\P}(\delta(x+S(n))\le \log^2 n).
$$
Combining this with Proposition~3 in \cite{EW25}, we infer that 
\begin{align}\label{eq:small_distance}
\widehat{\E}_x[W_x(n);\delta(x+S(n))\le \log^2 n]\to0
\quad \text{as }n\to\infty.
\end{align}
According to Theorem~2 in \cite{DW24} and to our Lemma~13,
\begin{equation}
\label{eq:V-asymp}
\lim_{\delta(y)\to\infty}\frac{V(y)}{u_{\pi/m}(y)}=1
\quad\text{and}\quad 
\lim_{\delta(y)\to\infty}\frac{h(y)}{u_{\pi/m}(y)}=1.
\end{equation}
Thus, there exists $\varepsilon_n\to0$ such that $|W_x(n)-1|\le\varepsilon_n$ on the event $\{\delta(x+S(n))>\log^2n\}$. Consequently,
$$
\widehat{\E}_x[W_x(n);\delta(x+S(n))> \log^2 n]\to1
\quad \text{as }n\to\infty.
$$
Combining this with \eqref{eq:small_distance}, we conclude that
$\lim\limits_{n\to\infty}\widehat{\E}_x[W_x(n)]=1$. But we know that $\widehat{\E}_x[W_x(n)]=W_x(0)$ for every $n$. This implies that $h(x)=V(x)$ for all $x$ such that $V(x)>0$. By Remark~\ref{rem:V-pos}, it remains to consider the point $\widehat{x}$ under the assumption that all non-zero jumps of $X^\phi$ are of the form $(a,-1)$
and $(-1,b)$. Under this assumption we have 
$$
h(\widehat{x})=\E[h(\widehat{x}+X);\widehat{x}+X\in K_{\pi/m}]
=h(\widehat{x})\P(X=0).
$$
Since $X$ is non-degenerate, this equation has only one solution: $h(\widehat{x})=0$.
Thus, the proof is complete.
\end{proof}
\section{Moments of $\tau_x$: proof of Theorem~\ref{thm:moments}}
We first derive upper bounds for the moments $\E[\tau_x^k]$. 
\begin{lemma}\label{E_tau_assympt}
        For every integer $k < \frac{\pi}{2\alpha}$ there exists a constant 
        $C$ such that 
        \begin{align*}
                \E\bigl[\tau_x^k\bigr]\le C\delta(x)|x|^{2k(p_\alpha-1)/p_\alpha},
                \quad x\in K_\alpha\cap R,
        \end{align*}
        where $p_\alpha=\pi/\alpha$.
\end{lemma}
\begin{proof}
        According to Theorem 3 in \cite{DW24}, for every fixed $x_0\in K_\alpha$ there exist $M>0$ and a constant $C$ such that 
        \begin{align*}
                \P(\tau_x > n) \le \frac{C u_\alpha(x+Mx_0)}{n^{p_\alpha/2}}.
        \end{align*}
        Combining this with the upper bound in \eqref{eq:BM2} and noting that there exists 
        $\delta_0>0$ such that $|x|\ge\delta(x)\ge\delta_0$ for all $x\in K_\alpha\cap R$, we conclude that 
        \begin{align}\label{eq:tau-ub}
                \P(\tau_x > n) \le C\frac{|x|^{p_\alpha - 1}\delta(x)}{n^{p_\alpha/2}},\quad x\in K_\alpha\cap R
        \end{align}
        with some $C > 0.$ 
        
        Set $p_k(z) := (z + 1)^k - z^k.$ Then we have
        \begin{align*}
                \sum\limits_{z = 0}^{n - 1}p_k(z) = n^k
        \end{align*}
        and 
        \begin{align*}
                \sum\limits_{n = 0}^\infty p_k(n)\P(\tau_x > n) &= \sum\limits_{n = 0}^\infty p_k(n)
                \Bigl[\sum\limits_{j = n + 1}^\infty\P(\tau_x = j)\Bigr]\\
                &= \sum\limits_{j = 1}^\infty \P(\tau_x = j)\sum\limits_{i = 0}^{j - 1}p_k(i) 
                = \sum\limits_{j = 1}^\infty j^k \P(\tau_x = j) \\
                &= \E[\tau_x^k].
        \end{align*}
         These equalities imply that for every $n_x\ge1$,
         $$
         \E[\tau_x^k]\le n_x^k+\sum_{n=n_x}^\infty p_k(n)\P(\tau_x>n).
         $$
         Noting that $p_k(n)\le k2^{k-1}n^{k-1}$ for all $n\ge1$ and applying \eqref{eq:tau-ub}, we obtain 
         \begin{align*}
         \sum_{n=n_x}^\infty p_k(n)\P(\tau_x>n)
         &\le C_1|x|^{p_\alpha-1}\delta(x)\sum_{n=n_x}^\infty n^{k-1-p_\alpha/2}\\
         &\le C_2 |x|^{p_\alpha-1}\delta(x)n_x^{k-p_\alpha/2}
         \end{align*}
         for every $k<\frac{p_\alpha}{2}=\frac{\pi}{2\alpha}$.
         Consequently,
         $$
         \E[\tau_x^k]\le n_x^k+C|x|^{p_\alpha-1}\delta(x)n_x^{k-p_\alpha/2}.
         $$
         Choosing now $n_x=1+\lfloor|x|^{2(p_\alpha-1)/p_\alpha}\rfloor$
         and using once again the bound $|x|\ge\delta(x)\ge\delta_0>0$, we obtain the desired inequality.
\end{proof}
Next we formulate a result which generalizes our construction of harmonic polynomials
to the case of polynomial right hand sides in the equation $L_X G=g$.
\begin{proposition}
\label{prop:moments}
Assume that all conditions of Theorem~\ref{thm:moments} hold. For every polynomial $f(x)$ of degree $n-2$ with $n<\frac{\pi}{\alpha}$ there exists a unique polynomial $F(x)$ of degree $n$
such that 
$$
L_XF(x)=f(x),\quad x\in K_\alpha\cap R
$$
and 
$$
F(x)=0,\quad x\in \partial K_\alpha\cap R.
$$
\end{proposition}
\begin{proof}
As usual, we decompose all the polynomials into homogeneous components:
$$
F(x)=\sum_{k=0}^n F_k(x)
\quad\text{and}\quad
f(x)=\sum_{k=0}^{n-2}f_k(x).
$$
Similar to the proof of Proposition~\ref{mx_uq}, we determine recursively the components $F_k$ starting with the leading part $F_n$. According to Lemma~\ref{taylor_exp},
$$
L_XF(x)=\frac{1}{2}\Delta F_n(x)+\sum_{k=0}^{n-1}\frac{1}{2}\Delta F_k(x)+g_F(x),
$$
where $\widehat{g}_F$ is a polynomial of degree $n-3$. Then we can rewrite the equation
$L_XF=f$ in the following way:
$$
\frac{1}{2}\Delta F_{n}(x)=f_{n-2}(x)+\sum_{k=0}^{n-3}f_k(x)
-\sum_{k=0}^{n-1}\frac{1}{2}\Delta F_k(x)-g_F(x).
$$
The polynomials $\frac{1}{2}\Delta F_{n}(x)$ and $f_{n-2}(x)$ are of degree $n-2$, all other polynomials on the right hand side are of lower degrees.
Therefore, we have equations
$$
\frac{1}{2}\Delta F_{n}(x)=f_{n-2}(x)
\quad\text{and}\quad 
\sum_{k=0}^{n-1}\frac{1}{2}\Delta F_k(x)=\sum_{k=0}^{n-3}f_k(x)-g_F(x).
$$
We already know that the boundary condition $F(x)=0$ for all $x\in \partial K_\alpha\cap R$ implies that every component $F_k$ is zero on $K_\alpha\cap R$.
Therefore, the highest component $F_n$ should solve
$$
\frac{1}{2}\Delta F_{n}(x)=f_{n-2}(x)
\quad\text{and}\quad
F_n(x)=0,\quad x\in \partial K_\alpha\cap R.
$$
In terms of linear equations we then have
\begin{align*}
                M_{n, b}  
                        \begin{pmatrix}
                                a_{n, 0}\\
                                \vdots\\
                                a_{n, n-2}\\
                                a_{n, n -1}\\
                                a_{n, n }\\
                        \end{pmatrix}
                        =
                        \begin{pmatrix}
                                p_{n-2, 0}\\
                                \vdots\\
                                p_{n-2, n-2}\\
                                0\\
                                0\\
                        \end{pmatrix},
        \end{align*}
where $F_{n}(x)=\sum_{i=0}^n a_{n,i}v_{n,i}$, $f_{n-2}(x)=\sum_{i=0}^{n-2}p_{n-2,i}v_{n-2,i}$ and $b=\tan\alpha$. The assumption $n<\pi/\alpha$ implies that $b<\tan(\pi/n)$. Then, by Proposition~\ref{mx_uq}, the matrix $M_{n,b}$ is invertible and the system of equations possesses unique solution.

Having $F_n$ and using once again Lemma~\ref{taylor_exp}, we conclude that the equation 
$L_XF=f$ is equivalent to 
$$
L_X(F-F_n)(x)=\sum_{k=0}^{n-3}f_k(x)-g_{F_n}(x).
$$
Additionally we have the boundary condition $(F-F_n)(x)=0$ for all 
$x\in \partial K_\alpha\cap R$. The polynomial $\sum_{k=0}^{n-3}f_k(x)-g_{F_n}(x)$ is of degree $n-3$ and the polynomial $F-F_n$ is of degree $n-1$. So, repeating the same argument $n-1$ times we determine all the components $F_k$.
\end{proof}
\begin{proof}[Proof of Theorem~\ref{thm:moments}.]
We first change cosmetically the definition of exit times, we set
$$
\tau_x:=\inf\{n\ge0:x+S(n)\notin K_\alpha\}.
$$
This definition allows us to consider starting points which do not belong to $K_\alpha$.
It is clear that $\tau_x=0$ for all $x\notin K_\alpha$. 

For every $l\ge1$ we set 
$$
G_l(x):=\E[\tau_x^l],\quad x\in K_\alpha\cap R.
$$

Let us first find an exact expression for the function $G_1(x)$. 
By the Markov property,
\begin{align*}
G_1(x)&=\E[\tau_x]=\sum_{y\in R}\E[\tau_x,x+S(1)=y]\\
&=1+\sum_{y\in K_\alpha\cap R}\E[\tau_y]
=1+\E[G_1(x+X);x+X\in K_\alpha].
\end{align*}
In other words,
$$
L_XG_1(x)=-1,\ x\in K_\alpha\cap R
\quad\text{and}\quad
G_1(x)=0,\ x \in \partial K_\alpha\cap R.
$$
By Proposition~\ref{prop:moments}, there exists a quadratic polynomial $F_1$ which  solves the same boundary problem. The fact that $F$ is zero on the boundary of $K_\alpha$ implies that $F_1$ is divisible by $x_2(x_2-x_1\tan\alpha)$. Recalling that $F_1$ is quadratic we conclude that $F_1(x)=Cx_2(x_2-x_1\tan\alpha)$. It is easy to see from the equation $L_XF_1(x)=-1$ that $C=-1$.
Therefore, $F_1(x)=x_2(x_1\tan\alpha-x_2)$. 
It is now clear that 
\begin{equation}
\label{eq:F-bound}
|F_1(x)|\le C |x|\delta(x),\quad x\in K_\alpha.
\end{equation}
Since $F_1$ and $G_1$ solve the same equation, their difference is harmonic for the walk killed at leaving $K_\alpha$:
$$
L_X(F_1-G_1)(x)=0,\ x\in K_\alpha\cap R
\quad\text{and}\quad
F_1(x)-G_1(x)=0,\ x \in \partial K_\alpha\cap R.
$$
Combining Lemma~\ref{E_tau_assympt} with $k=1$, \eqref{eq:F-bound} and \eqref{eq:V-asymp}, we conclude that
\begin{equation}
\label{eq:E1-asymp}
\lim_{\delta(x)\to\infty}\frac{F_1(x)-G_1(x)}{V(x)}=0.
\end{equation}
Furthermore, Lemma~\ref{E_tau_assympt} with $k=1$, \eqref{eq:F-bound} and Lemma~4 from \cite{EW25} imply
\begin{equation}
\label{eq:E1-bound}
\sup_{x:V(X)>0}\frac{|F_1(x)-G_1(x)|}{V(x)}<\infty.
\end{equation}
To show that $G_1(x)=F_1(x)$ we consider, similar to the proof of Theorem~\ref{thm:uniqueness}, the sequence
$$
W^{(1)}_x(n):=\frac{F_1(x+S(n))-G_1(x+S(n))}{V(x+S(n))},\quad n\ge0.
$$
The harmonicity of the functions in the numerator and in the denominator implies that
the sequence $W^{(1)}_x(n)$ is a martingale under the measure $\widehat{\P}_x$ and, in particular,
$$
\widehat{\E}_x[W^{(1)}_x(n)]=W^{(1)}_x(0)=\frac{F_1(x)-G_1(x)}{V(x)},\quad n\ge0.
$$
Furthermore, \eqref{eq:E1-bound} implies that this martingale is bounded.

Repeating the arguments from the proof of Theorem~\ref{thm:uniqueness} and using 
\eqref{eq:E1-asymp} instead of \eqref{eq:V-asymp}, we conclude that 
$\widehat{\E}_x[W_x(n)]\to 0$ as $n\to\infty$. Consequently, $F_1(x)=G_1(x)$ for every $x$ such that $V(x)>0$. As we already know, the function $V$ can be zero only in the point $\phi^{-1}(1,1)$. This happens if and only if all jumps of $X^\phi$ are of the form $(-1,a)$ or $(b,-1)$.  But in this case $\tau_{\phi^{-1}(1,1)}=1$ with probability one, i.e., $G_1(\phi^{-1}(1,1))$. On the other hand, this condition on jumps implies that $L_XF_1(\phi^{-1}(1,1))=-F_1(\phi^{-1}(1,1))$. Combining this with the equation
$L_XF_1=-1$, we conclude that $F_1(\phi^{-1}(1,1))=1.$ Thus, the functions $G_1$ and $F_1$ coincide on the whole set $K_\alpha\cap R$.

To derive the claimed properties of higher moments one can use the following recursive argument. 

Assume that Theorem~\ref{thm:moments} is proved for all $l\le k-1$, that is, every $G_l$ is polynomial of degree $2l$. Applying the Markov property, we obtain the following equality for the function $G_k$:
\begin{align*}
G_k(x)&:=\E[\tau_x^k]
=\sum_{y\in R}\E[(1+\tau_y)^k]\P(x+X=y)\\
&=1+\sum_{y\in K_\alpha\cap R}\sum_{l=1}^k{k\choose l}\E[\tau_y^l]\P(x+X=y)\\
&=1+\sum_{y\in K_\alpha\cap R}\E[\tau_y^k]\P(x+X=y)
+\sum_{y\in K_\alpha\cap R}\sum_{l=1}^{k-1}{k\choose l}\E[\tau_y^l]\P(x+X=y)\\
&=1+\E[G_k(x+X);x+X\in K_\alpha]
+\sum_{l=1}^{k-1}{k\choose l}\E[G_l(x+X);x+X\in K_\alpha].
\end{align*}
Since all functions $G_l$ equal zero on the boundary of $K_\alpha$,
\eqref{eq:no-overshoot} implies that
$$
\E[G_l(x+X);x+X\in K_\alpha]=L_XG_l(x)+G_l(x).
$$
Using these equalities, we infer that the function $G_k$ solves the equation 
\begin{align*}
L_X G_k(x)=-1-\sum_{l=1}^{k-1}{k\choose l}\left(G_l(x)+L_XG_l(x)\right),
\quad x\in K_\alpha\cap R.
\end{align*}
Clearly, the function
$g_k(x):=-1-\sum\limits_{l=1}^{k-1}{k\choose l}\left(G_l(x)+L_XG_l(x)\right)$ is a polynomial of degree $2k-2$. By Proposition~\ref{prop:moments}, there exists unique polynomial $F_k$ of degree $2k$ which solves
$$
L_XF_k(x)=g_k(x)
\quad\text{and}\quad
F_k(x)=0
\quad\text{for all}\quad
x\in\partial K_\alpha\cap R.
$$
The boundary condition implies that 
$$
F_k(x)=x_2(x_1\tan\alpha-x_2)\widetilde{F}_{k}(x),
$$
where $\widetilde{F}_{k}(x)$ is a polynomial of degree $2k-2$. This representation implies that 
$$
|F_k(x)|\le C\delta(x)|x|^{2k-1}.
$$
Combining this with Lemma~\ref{E_tau_assympt}, we infer that 
\begin{equation}
\label{eq:Ek-asymp}
\lim_{\delta(x)\to\infty}\frac{F_k(x)-G_k(x)}{V(x)}=0.
\end{equation}
This allows to repeat the arguments used in the case $k=1$ and to conclude that
$G_k(x)=F_k(x)$ for all $x\in K_\alpha\cap R$. Thus, the proof is complete.
\end{proof}
\begin{corollary}

$\quad$

    \begin{itemize}
        \item[(a).] For all $x\in\K_\alpha$ with $0 < \alpha < \pi$ one has 
        \begin{align*}
            \E\bigl[x_1 + S_1(\tau_x)\bigr] = x_1,\quad
            \E\bigl[x_2 + S_2(\tau_x)\bigr] = x_2.
        \end{align*}

        \item[(b).] For all $x\in\K_\alpha$ with $0 < \alpha < \pi/2$ one has
        \begin{align*}
            &\E\bigl[(x_1 + S_1(\tau_x))^2\bigr] = x_1^2 + x_2(x_1\tan\alpha-x_2),\\ 
            &\E\bigl[(x_2 + S_2(\tau_x))^2\bigr] = x_2^2 + x_2(x_1\tan\alpha-x_2). 
        \end{align*}
    \end{itemize}
\end{corollary}
\begin{proof}
    The assumption $\E X=0$ implies that the sequence $\{x_j+S_j(n)\}$ is a martingale. Applying the optional stopping theorem, we get, for every $n\ge1$,  
    \begin{align}
    \label{eq:A}
    \nonumber
        x_j&=\E\bigl[x_j + S_j(\tau_x \wedge n)\bigr]\\
           &=\E\bigl[x_j + S_j(\tau_x), \tau_x\le n\bigr]
              +\E\bigl[x_j + S_j(n), \tau_x > n\bigr].
    \end{align}
    Noting that $x_j+S_j(\tau_x)$ takes only nonnegative or only nonpositive values and applying the monotone convergence theorem, we conclude that
    \begin{align}
    \label{eq:B}  
    \lim_{n\to\infty}\E\bigl[x_j + S_j(\tau_x), \tau_x\le n\bigr]
    =\E\bigl[x_j + S_j(\tau_x)\bigr].
    \end{align}
    Fix  some $\varepsilon\in(0, \frac{p_\alpha - 1}{2})$ and split the the second summand on the right hand side of \eqref{eq:A} as follows: 
    \begin{align*}
        \E\bigl[|x_j + S_j(n)|, \tau_x > n\bigr]
        &\le |x_j|\P(\tau_x>n)+\E\bigl[|S_j(n)|; |S_j(n)| \le n^{1/2+\varepsilon}, 
        \tau_x > n\bigr]\\ 
        &\hspace{3cm}+\E\bigl[|S_j(n)|, |S_j(n)| > n^{1/2+\varepsilon}, \tau_x > n\bigr]\\
        &\le (|x_j|+n^{1/2+\varepsilon})\P(\tau_x>n)\\
        &\hspace{3cm}+\E\bigl[|S_j(n)|, |S_j(n)| > n^{1/2+\varepsilon}, \tau_x > n\bigr].
    \end{align*}
    Using Lemma 9 in \cite{EW25} with $p=q=2$ and
    $r=\frac{1+4\varepsilon}{4\varepsilon}$, we get 
    \begin{align*}
        \E\Bigl[|S_j(n)|; |S_j(n)| > n^{1/2+\varepsilon}, \tau_x > n\Bigr]
        \le
        C\left(n^{-\frac{1}{2} - \varepsilon}\sum\limits_{i = 1}^n 
        \P\bigl(\tau_x > i - 1\bigr) + n^{-\varepsilon}\right).
    \end{align*}
    Furthermore, according to Lemma 12 in \cite{EW25}, there exists a constant 
    $C'$ such that 
    \begin{align}\label{eq:sum_of prob}
        \sum_{i = 0}^{n - 1}\P(\tau_x > i)
        \le C'
        \left\{
        \begin{array}{ll}
            1, &p_\alpha>2,\\
            \log n, &p_\alpha=2,\\
            n^{1-p_\alpha/2}, &p_\alpha\in(1,2),
            \end{array}
        \right.
    \end{align}
    for some $C' > 0.$ This implies that if $\varepsilon$ is sufficiently close to $\frac{p_\alpha-1}{2}$ then
    \begin{align*}
        \lim\limits_{n\to\infty}\E\Bigl[|x_j + S_j(n)|; |S_j(n)| > 
        n^{\frac{1}{2} + \varepsilon}, \tau_x > n\Bigr] = 0.
    \end{align*}
    According to Theorem 3 in \cite{DW24},
    \begin{equation}
     \label{eq:tau-tail}   
     \P(\tau_x>n)\sim\varkappa\frac{V(x)}{n^{p_\alpha/2}}.
    \end{equation}
    This implies that 
    $$
    \lim_{n\to\infty}(|x_j|+n^{1/2+\varepsilon})\P(\tau_x>n)=0.
    $$
    As a result, we have 
    $$
    \lim_{n\to\infty} \E\bigl[x_j + S_j(n), \tau_x > n\bigr]=0
    $$
    Plugging this and \eqref{eq:B} into \eqref{eq:A}, we conclude that
    $$
    x_j=\E[x_j+S_j(\tau_x)].
    $$
    This is equivalent to the desired equality.
    
     To prove the part (b) we apply the optional stopping theorem  to the martingale $S^2_j(n) - n$:  
    \begin{align*}
        \E\bigl[S^2_j(\tau_x \wedge n) - (\tau_x \wedge n)\bigr] =0.
    \end{align*}
    By the monotone convergence theorem,
    $$
    \lim_{n\to\infty} \E[\tau_x \wedge n]=\E\tau_x.
    $$
    Consequently,
    \begin{align}\label{eq:D}
    \lim_{n\to\infty}\E S^2_j(\tau_x \wedge n)
    =\E\tau_x.
    \end{align}
    Using the monotone convergence theorem once again, we have 
    \begin{align}\label{eq:E}
    \lim_{n\to\infty}\E\left[S^2_j(\tau_x);\tau_x\le n\right]
    =\E[S^2_j(\tau_x)].
    \end{align}
    Similar to the proof of the part (a), we split 
    $\E\bigl[S^2_j(n), \tau_x > n\bigr]$ into two parts:
    \begin{align*}
        &\E\bigl[S^2_j(n), \tau_x > n\bigr]\\
        &\hspace{1cm}= 
        \E\bigl[S^2_j(n), |S_j(n)| \le n^{1/2+\varepsilon}, \tau_x > n\bigr]
        + \E\bigl[S^2_j(n), |S_j(n)| > n^{1/2+\varepsilon}, \tau_x > n\bigr]\\
         &\hspace{1cm}\le n^{1+2\varepsilon}\P(\tau_x>n)
         + \E\bigl[S^2_j(n), |S_j(n)| > n^{1/2+\varepsilon}, \tau_x > n\bigr].
    \end{align*}
    In view of \eqref{eq:tau-tail},
    \begin{align}
    \label{eq:F}
    \lim_{n\to\infty} n^{1+2\varepsilon}\P(\tau_x>n)=0
    \end{align}
    for every $\varepsilon\in(0,\frac{p_\alpha-2}{4}).$
    Furthermore, applying Lemma 9 in \cite{EW25}, we get, for every $q\in(2,p_\alpha]$,
    \begin{align*}
        &\E\Bigl[S^2_j(n);|S_j(n)| > n^{1/2+\varepsilon}, \tau_x > n\Bigr]\\
        &\le C(r,q)\left( n^{(1/2+\varepsilon)(2-q)}\sum\limits_{i = 1}^n 
        \P\bigl(\tau_x > i - 1\bigr) + n^r\frac{2r}{2r - 1}
        n^{(1/2+\varepsilon)(2-2r)}\right).
    \end{align*}
    Using \eqref{eq:sum_of prob} for $p_\alpha>2$ and choosing
    $r>\frac{1+2\varepsilon}{2\varepsilon}$, we conclude that
    \begin{align*}
        \lim\limits_{n\to\infty}\E\Bigl[S^2_j(n); |S_j(n)| > 
        n^{\frac{1}{2} + \varepsilon}, \tau_x > n\Bigr] = 0.
    \end{align*}
    Combining this relation with \eqref{eq:F}, we obtain 
    $$
    \lim_{n\to\infty}[S_j^2(n);\tau_x>n]=0.
    $$
    Plugging this and \eqref{eq:E} into \eqref{eq:D}, we finally arrive at
    $$
    \E[S_j^2(\tau_x)]=\E[\tau_x].
    $$
    Using now part (a) and Theorem~\ref{thm:moments}, we have
    \begin{align*}
    \E(x_j+S_j(\tau_x))^2
    &=x_j^2+2x_j\E S_j(\tau_x)+\E S_j^2(\tau_x)  \\
    &=x_j^2+x_2(x_1\tan\alpha-x_2). 
    \end{align*}
    Thus, the proof of the corollary is finished.
\end{proof}

\section{An alternative approach to the construction of harmonic polynomials.}
Behind our proof of Theorem~\ref{thm:harmonic} stays a rather clear idea. To construct a harmonic polynomial for the walk $S(n)$ killed at leaving the cone $K_{\pi/m}$ we start with the function $h_m(x)$ which is harmonic for the Brownian killed at leaving the same cone. If the function $L_Xh_m$
is zero then the function $h_m$ is harmonic also for $S(n)$. And if $L_Xh_m$ is not zero then we construct a first correction $h_{m-1}$ which  
is chosen so that $-\Delta  h_{m-1}$ equals to the homogeneous part of degree $m-3$ of $L_Xh_m$. In other words, $h_{m-1}$ eliminates a part of $L_Xh_m$ with highest degree. Then, the next correction $h_{m-2}$ eliminates all highest degrees in $L_X(h_m+h_{m-1})$, and so on. The corrections $h_n$ are constructed by solving an appropriate system of linear equations.

In this section we describe an alternative numerical approach which allows one to construct corrections $h_{n+2}$, $n=0,1,\ldots,m-3$. As we have mentioned before, we need to eliminate linear combinations of monomials $x_1^j x_2^{k}$ with $j + k = n$. Thus, it is sufficient to describe the elimination of every monomial.
We split this step into two parts.
First, we construct a polynomial $f_{j, k}$ such that 
$\Delta f_{j, k} = x_1^j x_2^{k}.$ Second, we find a polynomial 
$g_{j, k}$ such that
\begin{align*}
    \begin{cases}
        \Delta g_{j, k}=0\\
        g_{j, k}(x_{1},0)=-f_{j, k}(x_1, 0), \\
        g_{j, k}(x_{1}, x_{1}\tan({\pi/m})) = -f_{j, k}(x_{1}, x_{1}\tan({\pi/m})).
    \end{cases}
\end{align*}
Then the polynomial $F_{j, k} := f_{j, k} + g_{j, k}$ satisfies
\begin{align*}
    \begin{cases}
        \Delta F_{j, k} = x_{1}^{j}\, x_{2}^{k}, \\
        F_{j, k}(x_{1},0) = 0, \\
        F_{j, k}(x_{1}, x_{1}\tan{\pi/m}) = 0.
    \end{cases}
\end{align*}
This means that $F_{j,k}$ eliminates $x_1^jx_2^{k}$.

\vspace{6pt}

\textbf{Construction of $f_{j, k}$.}
To construct $f_{j, k}$ we use the polar coordinates:
\begin{align*}
    x_{1} = r\cos\beta, 
    \qquad
    x_{2} = r\sin\beta.    
\end{align*}
Then 
\begin{align*}
    x_{1}^{j} x_{2}^k = r^{j + k} \cos^{j}\beta\sin^k\beta.    
\end{align*}
and the Laplacian is given by 
\begin{align*}
    \Delta 
    = \frac{\partial^{2}}{\partial r^{2}}
    + \frac{1}{r} \frac{\partial}{\partial r}
    + \frac{1}{r^{2}} \frac{\partial^{2}}{\partial \beta^{2}}.    
\end{align*}

We want to find a solution of the form
\begin{align*}
    f_{j, k}(r,\beta)
    = r^{n + 2} \sum_{l = 0}^{n} 
    \bigl(\kappa_{l}\cos(l\beta) + \mu_{l}\sin(l\beta) \bigr),    
\end{align*}

where $n = j + k$. Then
\begin{align}\label{polar_Delta}
    \Delta f_{j,k}(r,\beta) = r^{n}\sum_{l = 0}^{n}\bigl(
    (n + 2)^2 - l^{2}\bigr)\bigl( 
    \kappa_{l}\cos(l\beta) + \mu_{l}\sin(l\beta) 
    \bigr).    
\end{align}



Using the equalities 
\begin{align*}
    \cos x = \frac{e^{ix}+e^{-ix}}{2}, \qquad
    \sin x = \frac{e^{ix}-e^{-ix}}{2i},
\end{align*}
we obtain
\begin{align*}
    \cos^j x \sin^k x
   = \left(\frac{e^{ix}+e^{-ix}}{2}\right)^j
     \left(\frac{e^{ix}-e^{-ix}}{2i}\right)^k
   = \frac{1}{2^{j + k}\, i^{\,k}}
     (e^{ix}+e^{-ix})^j (e^{ix}-e^{-ix})^k.
\end{align*}
By the binomial formula,
\begin{align*}
    (e^{ix}+e^{-ix})^j
        = \sum_{w = 0}^j \binom{j}{w} e^{\,i(j - 2w)x},
    \qquad
    (e^{ix}-e^{-ix})^k
        = \sum_{s = 0}^k \binom{k}{s} (-1)^s e^{\,i(k - 2s)x}.
\end{align*}

Thus
\begin{align}
\label{eq:interm.step}
    \cos^j x \sin^k x
    = \frac{1}{2^{j + k} i^{k}}\sum_{w = 0}^j \sum_{s = 0}^k
    \binom{j}{w}\binom{k}{s}(-1)^s e^{\,i \,(j + k - 2(w + s))x}.
\end{align}

Since $i^{-k}=e^{-ik\pi/2}$, we have
\begin{align*}
    i^{-k} e^{\,i(j + k - 2(w + s))x}
    = e^{\, i\big( (j + k - 2(w + s))x - k\pi/2 \big)}.
\end{align*}

Taking real parts on the right hand side of \eqref{eq:interm.step}, we obtain
\begin{align*}
    \cos^j x \sin^k x
    = \frac{1}{2^{j + k}}\sum_{w = 0}^j \sum_{s = 0}^k
    \binom{j}{w}\binom{k}{s}(-1)^s
    \cos\Bigl((j + k - 2(w + s))x - \tfrac{j\pi}{2}\Bigr).
\end{align*}

Using this equality and (\ref{polar_Delta}), we can write the equation
$\Delta f_{j,k}=x_1^jx_2^k$ in the following form:
\begin{align}
\label{eq:coeff}
\nonumber
    \sum_{l = 0}^{n}
    \bigl((n + 2)^2 - l^{2}\bigr)
    &\bigl(\kappa_{l}\cos(l\beta) + \mu_{l}\sin(l\beta) \bigr)\\
    &=
    \frac{1}{2^{n}}\sum_{w = 0}^j \sum_{s = 0}^k
    \binom{j}{w}\binom{k}{s}(-1)^s
    \cos\Bigl( (n - 2(w + s))\beta - \tfrac{k\pi}{2} \Bigr).
\end{align}
Notice that 
the term 
$\cos\Bigl((n - 2(w + s))\beta - \tfrac{k\pi}{2}\Bigr)$
is either $\cos\!\Bigl((n - 2(w + s))\beta \Bigr) $
or $\sin\Bigl((n - 2(w + s))\beta\Bigr)$. 
Since the functions $\sin(l\beta)$ and $\cos(l\beta)$ are linearly independent, \eqref{eq:coeff} uniquely determines the numbers $\kappa_l,\mu_l$. 
Furthermore, noticing that $n - 2(w + s)$ 
has the same parity as $n$, we infer that $\kappa_l,\mu_l$ might be non-zero only 
if $l$ has the same parity as $n$. This allows us to obtain the polynomial 
when we get back to Euclidian coordinates. 
Indeed, for Chebyshev polynomials $T_l$ and $U_l$ 
we further write 
\begin{align*}
    r^n\cos(l\beta) = r^{n - l} r^l T_l(\cos \beta)
    =(x_1^2+x_2^2)^\frac{n - l}{2}r^l T_l(\cos \beta). 
\end{align*}
Here 
$(x_1^2+x_2^2)^\frac{n - l}{2}$ 
is a polynomial since $n$ and $l$ have the same parity. 
Using the identity $\cos^2\beta + \sin^2\beta=1$ we can 
further rewrite $r^lT_l(\cos \beta)$ as a polynomial 
of $x_1,x_2$ of degree $l$. 
Similarly, 
\begin{align*}
    r^n\sin(l\beta) = r^{n - l} r^l\sin(\beta) U_{l - 1}(\sin \beta).
\end{align*}
Hence  we have a homogeneous polynomial $f_{j, k}$ of degree $n+2$ that 
satisfies 
\begin{align*}
    \Delta f_{j, k} = x_1^jx_2^k.
\end{align*}

\vspace{6pt}

\textbf{Construction of $g_{j, k}$.}
Since $f_{j,k}$ is homogeneous of degree $n+2$, we have
\begin{align*}
    f_{j, k}(x_1,0) = A x_1^{n + 2} 
\end{align*}
for some constant $A.$ Furthermore, on the line $x_2=(\tan \alpha)x_1$ we have 
\begin{align*}
    f_{j, k}(x_1,(\tan \alpha)x_1) = B x_2^{n + 2} 
\end{align*}
for some constant $B$.
Observe that the harmonic function
\(g_1(r,\beta)= r^{n + 2}\sin\!\bigl((n + 2)\beta\bigr)\), written in Euclidian coordinates,
vanishes on the line \( x_2 = 0 \) and satisfies
\begin{align*}
    g_1(x_1, (\tan \alpha)\, x_1) = \tilde A\, x_2^{\,n + 2}.
\end{align*}
Since \( n + 2 < m \), we conclude that \( \tilde A \neq 0 \).
 
Similarly, we can construct a harmonic function \( g_2 \) that
vanishes on the line \( x_2 = (\tan \alpha)\, x_1 \) and is non-zero
on the line \( x_2 = 0 \).
Then
\begin{align*}
    g_{j, k} = \hat A g_1 + \hat B g_2,
\end{align*}
where the constants \( \hat A \) and \( \hat B \) are chosen appropriately,
satisfies
\begin{align*}
    g_{j, k}(x_1,x_2) = -f_{j, k}(x_1,x_2)
\end{align*}
on the lines \( x_2 = 0 \) and \( x_2 = (\tan \alpha)\, x_1 \).

\clearpage
\section{Appendix}
Below we give a pseudocode of an algorithm for numerical computation of coefficients of positive discrete harmonic polynomials.
    \begin{algorithm}[H]
        \captionof{algorithm}{Computation of harmonic polynomials}\label{cn_alg}
        \begin{algorithmic}
            \State $b := \tan{\frac{\pi}{m}},\quad 
            m_{even} := \lfloor \frac{m}{2} \rfloor$
            \For {$k\gets 0, m_{even}$}
                \State $a_{m, 2k} := 0$  
            \EndFor
            \State $m_{odd} := \lceil \frac{m - 3}{2}\rceil$
            \For{$k\gets 0, m_{odd}$}
                \State $a_{m, 2k + 1} := (-1)^k\binom{m}{2k + 1}$
            \EndFor
            \State $n\gets m - 1$
            \While{$n > 1$} 
                \For {$s \gets 0, n - 2$}
                    \State $c_{n - 2, s} := 0$
                    \For {$i \gets n + 1, m$}
                        \For {$j \gets 0, i$}
                            \State $c_{n - 2, s} \gets c_{n - 2, s} - \binom{j}{s}
                            \binom{i - j}{n - 2 - s}\E\bigl[X_2^{\max\{j - s, 0\}}
                            X_1^{\max\{i - j - n + 2 + s, 0\}}\bigr]\cdot a_{i, j}$
                        \EndFor
                    \EndFor
                \EndFor
                \State $n_{even} := \lfloor\frac{n}{2}\rfloor$
                \State $a_{n, 0}\gets 0$
                \For {$k \gets 1, n_{even}$}
                    \State $a_{n, 2k} \gets \bigl(c_{n - 2, 2k - 2} - \binom{n - 2k + 2}{2}
                    \cdot a_{n, 2k - 2}\bigr)\cdot \frac{1}{\binom{2k}{2}}$
                \EndFor
                \State $n_{odd} := \lceil\frac{n - 3}{2}\rceil,
                \quad r_{n - 2} := 0$
                \For {$k \gets 0, n_{even}$}
                    \State $r_{n - 2}\gets -b^{2k}\cdot a_{n, 2k}$ 
                \EndFor
                \State $\lambda_{n, 1} := \frac{-b}{\binom{n - 1}{2}}$
                \For {$k \gets 1, n_{odd}$}
                    \State $r_{n - 2}\gets r_{n - 2} + \lambda_{n, 2k - 1}
                    \cdot c_{n - 2, 2k - 1}$
                    \If {$k + 1 \le n_{odd}$}
                        \State $\lambda_{n, k + 1} := (-1)\cdot \frac{b^{2k + 1} + 
                        \lambda_{n, k}\cdot \binom{2k + 1}{2}}{\binom{n - 2k - 1}{2}}$
                    \EndIf
                \EndFor
                \State $C(n) :=\frac{1}{{n\choose 2n_{odd}+1}},\quad
                \mu_{n, n_{odd}} = \frac{-\binom{2n_{odd} + 1}{2}}{C(n)\cdot \
                Q_n(1, b)}$
                \State $c_{n - 2, 2n_{odd} - 1} \gets c_{n - 2, 2n_{odd} - 1} + 
                r_{n - 2}\cdot\mu_{n, n_{odd}}$
                \For {$k \gets n_{odd} - 1, 1$}
                    \State $\mu_{n, k} := \frac{-\binom{2k}{2}}{\binom{n - 2k - 1}{2}}$
                    \State $c_{n - 2, 2k - 1}\gets c_{n - 2, 2k - 1} + 
                    c_{n - 2, 2k + 1}\cdot\mu_{n, k}$
                \EndFor
                \For {$k \gets 1, n_{odd} - 1$}
                    \State $a_{n, 2k + 1} := \frac{c_{n - 2, 2k + 1}}{\binom{n - 1 - 2k}{2}}$
                \EndFor
                \State $a_{n, 2n_{odd} + 1} := \frac{r_{n - 2}}{C(n)\cdot Q_n(1, b)}$
            \EndWhile
        \end{algorithmic}
    \end{algorithm}

\end{document}